\documentclass[12pt,a4paper]{article}
\usepackage{amsfonts}
\usepackage{amssymb}
\usepackage{mathrsfs}
\usepackage{amsmath}
\usepackage{amsthm}
\usepackage{enumerate}
\usepackage{cite}
\usepackage{xypic}
\usepackage{amscd}
\allowdisplaybreaks
\newtheorem{lemma}{Lemma}[section]
\newtheorem{theorem}[lemma]{Theorem}
\newtheorem{remark}[lemma]{Remark}
\newtheorem{proposition}[lemma]{Proposition}
\newtheorem{definition}[lemma]{Definition}
\newtheorem{corollary}[lemma]{Corollary}

\begin{document}

\title{\textbf{3-ary Hom-Lie superalgebras induced Hom-Lie superalgebras}
\author{ Baoling Guan$^{1}$,  Liangyun Chen$^{2}$, Bing Sun$^{2}$
 \date{{\small {$^1$ College of Sciences, Qiqihar
University, Qiqihar 161006, China}\\{\small {$^2$ School of Mathematics and Statistics, Northeast Normal
 University,\\
Changchun 130024, China}}}}}}

\date{ }
\maketitle
\begin{quotation}
\small\noindent \textbf{Abstract}:
The purpose of this paper is to study the relationships between a Hom-Lie superalgebra and its induced 3-ary-Hom-Lie superalgebra. We provide an overview of the theory
and explore the structure properties such as ideals, center, derived series, solvability,
nilpotency, central extensions, and the cohomology.


\renewcommand{\thefootnote}{\fnsymbol{footnote}}
 \footnote[0]{Corresponding
author(L. Chen): chenly640@nenu.edu.cn.}
 \footnote[0]{Supported by NNSF
of China (Nos. 11171055 and 11471090).}
\end{quotation}
\setcounter{section}{0}

\section{Introduction}

Ternary Lie algebras appeared first in Nambu¡¯s generalization of Hamiltonian mechanics
\cite{ny} which use a generalization of Poisson algebras with a ternary bracket. The algebraic
formulation of Nambu mechanics is due to Takhtajan while the structure of $n$-Lie algebra was
studied by Filippov \cite{fvt},  then completed by Kasymov in \cite{kshm}, where solvability and nilpotency
properties were studied.
The cohomology of $n$-Lie algebras, generalizing the Chevalley-Eilenberg Lie algebras cohomology,
was first introduced by Takhtajan \cite{tla} in its simplest form, later a complex adapted
to the study of formal deformations was introduced by Gautheron \cite{gph}, then reformulated
by Daletskii and Takhtajan  \cite{dt} using the notion of base Leibniz algebra of an $n$-Lie algebra. The theory of
cohomology for  first-class $n$-Lie algebras can be found in \cite{ma}.
In \cite{akms}, the structure and cohomology of 3-Lie algebras induced by Lie algebras have been
investigated.
In \cite{ams3}, generalizations of $n$-ary algebras of Lie type and associative type by twisting
the identities using linear maps have been introduced. These generalizations include $n$-ary
Hom-algebra structures generalizing the $n$-ary algebras of Lie type including $n$-ary Nambu
algebras, $n$-ary Nambu-Lie algebras and $n$-ary Lie algebras, and $n$-ary algebras of associative
type including $n$-ary totally associative and $n$-ary partially associative algebras. In \cite{ams1},
a method was demonstrated of how to construct ternary multiplications from the binary
multiplication of a Hom-Lie algebra, a linear twisting map, and a trace function satisfying
certain compatibility conditions; and it was shown that this method can be used to construct
ternary Hom-Nambu-Lie algebras from Hom-Lie algebras. This construction was generalized
to $n$-Lie algebras and $n$-Hom-Nambu-Lie algebras in  \cite{ams2}.

The reference \cite{av} constructed super 3-Lie algebras(that is, 3-ary Lie superalgebras)  by super Lie
algebras(that is, Lie superalgebras). The reference \cite{aas} constructed $(n + 1)$-Hom-Lie algebras by $n$-Hom-Lie algebras.  Inspired by the references \cite{aas} and \cite{av},
the paper generalizes them to the case of Hom-superalgebra. Its aim is to study the relationships between a
Hom-Lie superalgebra and its
induced 3-ary-Hom-Lie superalgebra.   We explore the structure properties of objects such as ideals,
center, derived series, solvability, nilpotency, central extensions, and the cohomology.
In Section 2, we recall the basic notions for Hom-Lie superalgebras
and the construction of 3-ary-Hom-Lie superalgebras induced by Hom-Lie superalgebras, we also give a new
construction theorem for such algebras and give some basic properties. In Section 3, we define
the notion of solvability and nilpotency for 3-ary-Hom-Lie superalgebras and discuss solvability and
nilpotency of 3-ary-Hom-Lie superalgebras induced by Hom-Lie superalgebras. In Section 4, we
recall the definition and main properties of central extensions of Hom-Lie superalgebras, and
then we study central extensions of 3-ary-Hom-Lie superalgebras induced by Hom-Lie superalgebras.
The Section 5 is dedicated to study the corresponding cohomology.

\section{Preliminaries}

\subsection{3-ary Lie superalgebras}
 In \cite{av}, let $V=V_{\bar{0}}\oplus V_{\bar{1}}$ be a  finite-dimensional $\mathbb{Z}_{2}$-graded vector space. If $v\in V$ is a homogenous element, then its degree will be denoted by $|v|,$ where $|v|\in \mathbb{Z}_{2}$ and $\mathbb{Z}_{2}=\{\bar{0},\bar{1}\}.$
 Let $\mathrm{End}V$ be the $\mathbb{Z}_{2}$-graded vector space of endomorphisms of a $\mathbb{Z}_{2}$-graded vector space $V=V_{\bar{0}}\oplus V_{\bar{1}}.$
 The composition of two endomorphisms $a\circ b$ determines the structure of superalgebra in $\mathrm{End}V,$ and the graded binary commutator $[a,b]=a\circ b-(-1)^{|a||b|}b\circ a$ induces the structure of  Lie superalgebras  in $\mathrm{End}V.$ The supertrace of an endomorphism $a: V\rightarrow V$ can be defined
 by
\[
\mathrm{str}(a)=\left\{
\begin{array}{ll}
Tr(a|_{V_{0}})-Tr(a|_{V_{1}}),& \mbox{if}\ a \ \mbox{is even};\\
0,& \mbox{if}\ a\ \mbox{is odd}.
\end{array}
\right.
\]
For any endomorphisms $\sigma,\tau,$ it holds $\mathrm{str}([\sigma,\tau])=0.$

\begin{definition} {\rm\cite{{{CK}}}}\,  A $\mathbb{Z}_{2}$-graded vector space $\mathfrak{g}=\mathfrak{g}_{\bar{0}}\oplus \mathfrak{g}_{\bar{1}}$ is said to be a 3-ary Lie superalgebra, if it is endowed with  a
trilinear map (bracket) $[\cdot,\cdot,\cdot]:\mathfrak{g}\times \mathfrak{g}\times \mathfrak{g}\rightarrow\mathfrak{g}.$   If it satisfies the
following conditions:
\begin{align*}
&|[x_{1},x_{2},x_{3}]|=|x_{1}|+|x_{2}|+|x_{3}|;
\\&[x_{1},x_{2},x_{3}]=-(-1)^{|x_{1}||x_{2}|}[x_{2},x_{1},x_{3}],\,[x_{1},x_{2},x_{3}]=-(-1)^{|x_{2}||x_{3}|}[x_{1},x_{3},x_{2}];
\\&[x,y,[z,u,v]]=[[x,y,z],u,v]
\\&+(-1)^{|z|(|x|+|y|)}[z,[x,y,u],v]
+(-1)^{(|z|+|u|)(|x|+|y|)}[z,u,[x,y,v]].
\end{align*}
where $x,y,z,u,v\in \mathfrak{g}$ are homogeneous elements and $|x|$ is the $\mathbb{Z}_{2}$-degree of the homogeneous element $x$ in $\mathfrak{g}.$
\end{definition}

\subsection{$3$-ary multiplicative Hom-Lie superalgebras induced by Hom-Lie superalgebras}

\begin{definition} {\rm\cite{afa}}\,
 A Hom-Lie superalgebra is a triple $(\mathfrak{g}, [\cdot,\cdot], \alpha)$
consisting of a $\mathbb{Z}_{2}$-graded vector space $\mathfrak{g}=\mathfrak{g}_{\bar{0}}\oplus \mathfrak{g}_{\bar{1}},$ an even bilinear map $[\cdot,\cdot]: \Lambda^{2}\mathfrak{g}\rightarrow \mathfrak{g}$ and an even homomorphism
$\alpha: \mathfrak{g}\rightarrow \mathfrak{g}$ satisfying the following skew-supersymmetry and Hom-Jacobi identity:
$$[x, y]=-(-1)^{|x||y|}[y, x];$$
$$(-1)^{|x||z|}[\alpha(x), [y, z]]+(-1)^{|y||x|}[\alpha(y), [z, x]]+(-1)^{|z||y|}[\alpha(z), [x, y]]=0$$
where $x, y$ and $z$ are homogeneous elements in $\mathfrak{g},$ denote by $|x|$ the degree of a homogeneous element $x.$

 A $\mathbb{Z}_{2}$-graded vector space $I\subseteq \mathfrak{g}$ is a Hom-subalgebra
of $(\mathfrak{g},[\cdot,\cdot],\alpha)$  if $\alpha(I) \subseteq I$  and
$I$ is closed under
 the bracket operation  $[\cdot,\cdot]$, i.e.
$[I,I] \subseteq I$.

  A Hom-subalgebra $I$ is called a Hom-ideal of $\mathfrak{g}$  if
$[I,\mathfrak{g}]\subseteq I$.
\end{definition}

\begin{definition} {\rm\cite{{gbl}}}\,  A 3-ary-Hom-Lie superalgebra is a triple $(\mathfrak{g},[\cdot,\cdot,\cdot],\alpha)$ consisting of a $\mathbb{Z}_{2}$-graded vector space $\mathfrak{g}=\mathfrak{g}_{\bar{0}}\oplus \mathfrak{g}_{\bar{1}},$ an even
trilinear map (bracket) $[\cdot,\cdot,\cdot]:\mathfrak{g}\times \mathfrak{g}\times \mathfrak{g}\rightarrow\mathfrak{g}$ and an algebra endomorphism $\alpha=(\alpha_{1},\alpha_{2}), \alpha_{i}: \mathfrak{g}\rightarrow\mathfrak{g}$ is an even
linear map.  If it satisfies the
following conditions:
\begin{align*}
&|[x_{1},x_{2},x_{3}]|=|x_{1}|+|x_{2}|+|x_{3}|;
\\&[x_{1},x_{2},x_{3}]=-(-1)^{|x_{1}||x_{2}|}[x_{2},x_{1},x_{3}],\,[x_{1},x_{2},x_{3}]=-(-1)^{|x_{2}||x_{3}|}[x_{1},x_{3},x_{2}];
\\&[\alpha(x),\alpha(y),[z,u,v]]=[[x,y,z],\alpha(u),\alpha(v)]
\\&+(-1)^{|z|(|x|+|y|)}[\alpha(z),[x,y,u],\alpha(v)]
+(-1)^{(|z|+|u|)(|x|+|y|)}[\alpha(z),\alpha(u),[x,y,v]].
\end{align*}
where $|x|$ is the $\mathbb{Z}_{2}$-degree of the homogeneous element $x$ in $\mathfrak{g}.$
\end{definition}
Suppose that $(\mathfrak{g}, [\cdot,\cdot,\cdot],\alpha_{1}, \alpha_{2})$ is a 3-ary-Hom-Lie superalgebra, if $\alpha=(\alpha_{1},\alpha_{2})$
and $\alpha_{1}=\alpha_{2}=\alpha,$ it is satisfied
$$\alpha[x_{1},x_{2},x_{3}]=[\alpha(x_{1}),\alpha(x_{2}),\alpha(x_{3})], \forall x_{1},x_{2},x_{3}\in \mathfrak{g}.$$
Then $(\mathfrak{g}, [\cdot,\cdot,\cdot]_{\rho},\alpha_{1}, \alpha_{2})$ is called as multiplicative.

\begin{definition}  {\rm\cite{afmasn}}\,
 A representation of the Hom-Lie superalgebra $(\mathfrak{g},[\cdot,\cdot],\alpha)$ on a $\mathbb{Z}_{2}$-graded vector space $V=V_{\bar{0}}\oplus V_{\bar{1}}$ with
respect to $\beta\in \mathfrak{gl}(V)_{\bar{0}}$ is an even linear map $\rho: \mathfrak{g}\rightarrow\mathfrak{gl}(V),$ such that for all $x,y\in \mathrm{hg}(\mathfrak{g}),$ the following
equalities are satisfied:
\begin{eqnarray*}\rho(\alpha(x))\circ\beta &=& \beta\circ\rho(x);\\
\rho([x,y])\circ\beta &=& \rho(\alpha(x))\circ\rho(y)-(-1)^{|x||y|}\rho(\alpha(y))\circ\rho(x).
\end{eqnarray*}
\end{definition}

Let $(\mathfrak{g}, [\cdot,\cdot],\alpha)$ be a Hom-Lie superalgebra and $\rho: \mathfrak{g}\rightarrow \mathrm{gl}(V)$ is a representation of $(\mathfrak{g}, [\cdot,\cdot],\alpha).$  For any homogeneous element $x_{1},x_{2}\in \mathfrak{g},$ we define the 3-ary bracket by
\begin{align*}[x_{1},x_{2},x_{3}]_{\rho}=&\mathrm{str}\rho(x_{1})[x_{2},x_{3}]-(-1)^{|x_{1}||x_{2}|}
\mathrm{str}\rho(x_{2})[x_{1},x_{3}]
+(-1)^{|x_{3}|(|x_{1}|+|x_{2}|)}\mathrm{str}\rho(x_{3})[x_{1},x_{2}].
\end{align*}
\begin{theorem} \label{th1.24}
Let $(\mathfrak{g}, [\cdot,\cdot],\alpha)$ be a  Hom-Lie superalgebra. $\rho: \mathfrak{g}\rightarrow \mathrm{gl}(V)$ is a representation of $(\mathfrak{g}, [\cdot,\cdot],\alpha),$ $\beta: V\rightarrow V$ be an even linear map. If the following conditions are satisfied:
\begin{eqnarray}\mathrm{str}\rho(x)
\mathrm{str}\rho(\alpha(y))&=&\mathrm{str}\rho(x)\mathrm{str}\rho(\alpha(y));\label{eq1}\\
\mathrm{str}\rho(x)
\mathrm{str}\rho(\beta(y))&=&\mathrm{str}\rho(x)\mathrm{str}\rho(\beta(y));\label{eq2}\\
\mathrm{str}\rho(\alpha(x))\beta(y)&=&\mathrm{str}\rho(\beta(x))\alpha(y).\label{eq3}
\end{eqnarray}
for all $x,y\in \mathfrak{g}.$
Then $(\mathfrak{g}, [\cdot,\cdot,\cdot]_{\rho},\alpha,\beta)$ is a 3-ary-Hom-Lie superalgebra, and we say that it is induced by $(\mathfrak{g},[\cdot,\cdot],\alpha),$ it is denoted by $\mathfrak{g}_{\rho}.$
\end{theorem}
\begin{proof}
For $x_{2},x_{1},x_{3}\in\mathfrak{g},$ we have
 \begin{align*}
[x_{2},x_{1},x_{3}]_{\rho}=&\mathrm{str}\rho(x_{2})[x_{1},x_{3}]-(-1)^{|x_{1}||x_{2}|}
\mathrm{str}\rho(x_{1})[x_{2},x_{3}]
+(-1)^{|x_{3}|(|x_{1}|+|x_{2}|)}\mathrm{str}\rho(x_{3})[x_{2},x_{1}]
\\=&-(-1)^{|x_{1}||x_{2}|}\{\mathrm{str}\rho(x_{1})[x_{2},x_{3}]-(-1)^{|x_{1}||x_{2}|}\mathrm{str}\rho(x_{2})[x_{1},x_{3}]
\\&+(-1)^{|x_{3}|(|x_{1}|+|x_{2}|)}\mathrm{str}\rho(x_{3})[x_{1},x_{2}]\}
\\=&-(-1)^{|x_{1}||x_{2}|}[x_{1},x_{2},x_{3}]_{\rho}.
\end{align*}
Similarly, $[x_{1},x_{3},x_{2}]_{\rho}=-(-1)^{|x_{2}||x_{3}|}[x_{1},x_{2},x_{3}]_{\rho}.$
Hence $[\cdot,\cdot,\cdot]_{\rho}$ is an skew-supersymmetric map.  It is clear that it is trilinear by construction, one only has to prove that
the Hom-Nambu identity is fulfilled. Expanding the Hom-Nambu identity, that is,
\begin{align*}
&[\alpha(x),\beta(y),[z,u,v]_{\rho}]_{\rho}=[[x,y,z]_{\rho},\alpha(u),\beta(v)]_{\rho}
\\&+(-1)^{|z|(|x|+|y|)}[\alpha(z),[x,y,u]_{\rho},\beta(v)]_{\rho}
+(-1)^{(|z|+|u|)(|x|+|y|)}[\alpha(z),\beta(u),[x,y,v]_{\rho}]_{\rho}.
\end{align*}
gives 24 different terms. Six of these can be grouped into three  pairs as follows:
\begin{align*}
&(-1)^{(|u|+|z|)(|x|+|y|)}[[x,y],\beta(v)](\mathrm{str}\rho(u)\mathrm{str}\rho(\alpha(z))-\mathrm{str}\rho(z)\mathrm{str}\rho(\alpha(u)),
\\&(-1)^{(|u|+|v|+|z|)(|x|+|y|)}[\mathrm{str}\rho(z)(\mathrm{str}\rho(\alpha(v))\beta(u)-(-1)^{|v||u|}\mathrm{str}\rho(\beta(v))\alpha(u)),[x,y]],
\\&(-1)^{(|u|+|v|+|z|)(|x|+|y|)+|v||z|}[(\mathrm{str}\rho(u)\mathrm{str}\rho(\beta(v))-(-1)^{(|v|+|u|)|z|}\mathrm{str}\rho(v)\mathrm{str}\rho(\beta(u)))
[\alpha(z),[x,y]],
\end{align*}
which all vanish separately by (\ref{eq1})-(\ref{eq3}). The remaining 18 terms can be grouped into six  triples of the type
\begin{align*}&-(-1)^{|u||z|}\mathrm{str}\rho(u)\mathrm{str}\rho(\alpha(x))[\beta(y),[z,v]]
+(-1)^{|u|(|y|+|z|)}\mathrm{str}\rho(x)\mathrm{str}\rho(\alpha(u))[[y,u],\beta(v)]
\\&+(-1)^{(|z|+|v|)(|x|+|y|)+|z||u|}\mathrm{str}\rho(x)\mathrm{str}\rho(\beta(u))[\alpha(z),[y,v]].
\end{align*}
By using (\ref{eq3}), one can rewrite this term as
\begin{align*}&-(-1)^{|u||z|}\mathrm{str}\rho(u)\mathrm{str}\rho(\beta(x))[\alpha(y),[z,v]]
+(-1)^{|u|(|y|+|z|)}\mathrm{str}\rho(x)\mathrm{str}\rho(\alpha(u))[[y,z],\beta(v)]
\\&+(-1)^{(|z|+|u|)(|x|+|y|)+|z||u|}\mathrm{str}\rho(x)\mathrm{str}\rho(\beta(u))[\alpha(z),[y,v]],
\end{align*}
and by using (\ref{eq2}) and the Hom-Jacobi identity, one sees that this term vanishes. The remaining five triples of terms
 can be shown to vanish in an analogous way. Hence the Hom-Nambu identity is satisfied.
\end{proof}
\begin{proposition} \label{pro1.26} Let $(A, [\cdot,\cdot]^{'},\alpha_{1})$ and $(B, [\cdot,\cdot]^{''},\beta_{1})$ be Hom-Lie superalgebras. Let $\rho_{1}: A\rightarrow \mathrm{End}(V_{1})$(resp. $\rho_{2}: B\rightarrow \mathrm{End}(V_{2})$) is the reprensentations of $(A, [\cdot,\cdot]^{'},\alpha_{1})$(resp. $(B, [\cdot,\cdot]^{''},\beta_{1}).$ Let $\alpha_{2}$(resp. $\beta_{2})$ be an even linear map $\alpha_{2}: A\rightarrow A$(resp. $\beta_{2}: B\rightarrow B$).
Set $(A, [\cdot,\cdot,\cdot]_{\rho_{1}}^{'},\alpha^{'}=(\alpha_{1},\alpha_{2}))$ (resp. $(B, [\cdot,\cdot,\cdot]_{\rho_{2}}^{''},\beta^{'}=(\beta_{1},\beta_{2})))$ to be the induced $3$-ary-Hom-Lie
superalgebras.  Let $f : A\rightarrow B$ be a Hom-Lie superalgebra homomorphism satisfying $\mathrm{str}\rho_{2}\circ f=\mathrm{str}\rho_{1}$
and $f\circ \alpha_{i}=\beta_{i}\circ f(i=1,2),$ then $f$ is a $3$-ary-Hom-Lie
superalgebras homomorphism of the induced
algebras.
\end{proposition}
\begin{proof}
For all $x_{1},x_{2},x_{3}\in A,$ we have:
\begin{align*}f([x_{1},x_{2},x_{3}]_{\rho_{1}}^{'})=&f\{\mathrm{str}\rho_{1}(x_{1})[x_{2},x_{3}]^{'}-(-1)^{|x_{1}||x_{2}|}
\mathrm{str}\rho_{1}(x_{2})[x_{1},x_{3}]^{'}
\\&+(-1)^{|x_{3}|(|x_{1}|+|x_{2}|)}\mathrm{str}\rho_{1}(x_{3})[x_{1},x_{2}]^{'}\}
\\=&\mathrm{str}\rho_{1}(x_{1})[f(x_{2}),f(x_{3})]^{''}-(-1)^{|x_{1}||x_{2}|}
\mathrm{str}\rho_{1}(x_{2})[f(x_{1}),f(x_{3})]^{''}
\\&+(-1)^{|x_{3}|(|x_{1}|+|x_{2}|)}\mathrm{str}\rho_{1}(x_{3})[f(x_{1}),f(x_{2})]^{''}
\\=&\mathrm{str}\rho_{2}(f(x_{1}))[f(x_{2}),f(x_{3})]^{''}-(-1)^{|x_{1}||x_{2}|}
\mathrm{str}\rho_{2}(f(x_{2}))[f(x_{1}),f(x_{3})]^{''}
\\&+(-1)^{|x_{3}|(|x_{1}|+|x_{2}|)}\mathrm{str}\rho_{2}(f(x_{3}))[f(x_{1}),f(x_{2})]^{''}
\\=&[f(x_{1}),f(x_{2}),f(x_{3})]_{\rho_{2}}^{''}.
\end{align*}
We have $f\circ \alpha_{1}=\beta_{1}\circ f,$ since $f$ is a Hom-Lie superalgebra homomorphism,
and we also have $f\circ \alpha_{2}=\beta_{2}\circ f.$ This means that $f$ is a $3$-ary-Hom-Lie
superalgebras homomorphism of
$(A, [\cdot,\cdot,\cdot]_{\rho_{1}}^{'},\alpha^{'}=(\alpha_{1},\alpha_{2}))$
and $(B, [\cdot,\cdot,\cdot]_{\rho_{2}}^{''},\beta^{'}=(\beta_{1},\beta_{2})).$
\end{proof}
The new theorem for constructing $3$-ary-Hom-Lie
superalgebras induced by Hom-Lie superalgebras
can be formulated as follows:

\begin{theorem} \label{th1.27}
Let $(\mathfrak{g}, [\cdot,\cdot],\alpha)$ be a multiplicative Hom-Lie superalgebra and $\rho: \mathfrak{g}\rightarrow \mathrm{gl}(V)$ is a representation of $(\mathfrak{g}, [\cdot,\cdot],\alpha)$ satisfying $\mathrm{str}\rho\circ \alpha=\mathrm{str}\rho,$ then $(\mathfrak{g}, [\cdot,\cdot,\cdot]_{\rho},\alpha)$ is a multiplicative 3-ary-Hom-Lie superalgebra.
\end{theorem}
\begin{proof} We know that $[\cdot,\cdot,\cdot]_{\rho}$ is an trilinear skew-supersymmetric map. We next show
that it satisfies the Hom-Nambu identity, that is,
\begin{align*}
&[\alpha(x),\alpha(y),[z,u,v]_{\rho}]_{\rho}=[[x,y,z]_{\rho},\alpha(u),\alpha(v)]_{\rho}
\\&+(-1)^{|z|(|x|+|y|)}[\alpha(z),[x,y,u]_{\rho},\alpha(v)]_{\rho}
+(-1)^{(|z|+|u|)(|x|+|y|)}[\alpha(z),\alpha(u),[x,y,v]_{\rho}]_{\rho}.
\end{align*}
Let $L$ be its left-hand side, and $R$ its right-hand
side:
\begin{align*}
&L=[\alpha(x),\alpha(y),[z,u,v]_{\rho}]_{\rho}
\\=&[\alpha(x),\alpha(y),\mathrm{str}\rho(z)[u,v]-(-1)^{|u||z|}
\mathrm{str}\rho(u)[z,v]+(-1)^{|v|(|z|+|u|)}\mathrm{str}\rho(v)[z, u]]_{\rho}
\\=&\mathrm{str}\rho(z)[\alpha(x),\alpha(y),[u,v]]_{\rho}-(-1)^{|u||z|}
\mathrm{str}\rho(u)[\alpha(x),\alpha(y),[z,v]]_{\rho}
\\&+(-1)^{|v|(|z|+|u|)}\mathrm{str}\rho(v)[\alpha(x),\alpha(y),[z, u]]_{\rho}.
\\=&\mathrm{str}\rho(z)\{\mathrm{str}\rho(\alpha(x))[\alpha(y),[u,v]]-(-1)^{|y||x|}
\mathrm{str}\rho(\alpha(y))[\alpha(x),[u,v]]\}
\\&-(-1)^{|u||z|}\mathrm{str}\rho(u)\{\mathrm{str}\rho(\alpha(x))[\alpha(y),[z,v]]-(-1)^{|y||x|}
\mathrm{str}\rho(\alpha(y))[\alpha(x),[z,v]]\}
\\&+(-1)^{|v|(|z|+|u|)}\mathrm{str}\rho(v)\{\mathrm{str}\rho(\alpha(x))[\alpha(y),[z,u]]-(-1)^{|y||x|}
\mathrm{str}\rho(\alpha(y))[\alpha(x),[z,u]]\}
\\=&\mathrm{str}\rho(z)\mathrm{str}\rho(\alpha(x))[\alpha(y),[u,v]]-(-1)^{|y||x|}\mathrm{str}\rho(z)
\mathrm{str}\rho(\alpha(y))[\alpha(x),[u,v]]
\\&-(-1)^{|u||z|}\mathrm{str}\rho(u)\mathrm{str}\rho(\alpha(x))[\alpha(y),[z,v]]
\\&+(-1)^{|u||z|+|y||x|}
\mathrm{str}\rho(u)\mathrm{str}\rho(\alpha(y))[\alpha(x),[z,v]]
\\&+(-1)^{|v|(|z|+|u|)}\mathrm{str}\rho(v)\mathrm{str}\rho(\alpha(x))[\alpha(y),[z,u]]
\\&-(-1)^{|v|(|z|+|u|)+|y||x|}
\mathrm{str}\rho(v)\mathrm{str}\rho(\alpha(y))[\alpha(x),[z,u]]
\\=&-(-1)^{|y|(|u|+|v|)}\mathrm{str}\rho(z)\mathrm{str}\rho(\alpha(x))[\alpha(u),[v,y]]
\\&-(-1)^{|v|(|y|+|u|)}\mathrm{str}\rho(z)\mathrm{str}\rho(\alpha(x))[\alpha(v),[y,u]]
\\&+(-1)^{|x|(|y|+|u|+|v|)}\mathrm{str}\rho(z)\mathrm{str}\rho(\alpha(y))[\alpha(u),[v,x]]
\\&+(-1)^{|x||y|+|v|(|u|+|x|)}\mathrm{str}\rho(z)\mathrm{str}\rho(\alpha(y))[\alpha(v),[x,u]]
\\&+(-1)^{|u||z|+|y|(|z|+|v|)}\mathrm{str}\rho(u)\mathrm{str}\rho(\alpha(x))[\alpha(z),[v,y]]
\\&+(-1)^{|u||z|+|v|(|z|+|y|)}
\mathrm{str}\rho(u)\mathrm{str}\rho(\alpha(x))[\alpha(v),[y,z]]
\\&-(-1)^{|u||z|+|x||y|+|x|(|z|+|v|)}\mathrm{str}\rho(u)\mathrm{str}\rho(\alpha(y))[\alpha(z),[v,x]]
\\&-(-1)^{|u||z|+|x||y|+|v|(|z|+|x|)}
\mathrm{str}\rho(u)\mathrm{str}\rho(\alpha(y))[\alpha(v),[x,z]]
\\&-(-1)^{(|v|+|y|)(|z|+|u|)}
\mathrm{str}\rho(v)\mathrm{str}\rho(\alpha(x))[\alpha(z),[u,y]]
\\&-(-1)^{|v|(|z|+|u|)+|u|(|y|+|z|)}
\mathrm{str}\rho(v)\mathrm{str}\rho(\alpha(x))[\alpha(u),[y,z]]
\\&+(-1)^{(|v|+|x|)(|z|+|u|)+|x||y|}
\mathrm{str}\rho(v)\mathrm{str}\rho(\alpha(y))[\alpha(z),[u,x]]
\\&+(-1)^{|v|(|z|+|u|)+|x||y|+|u|(|x|+|z|)}
\mathrm{str}\rho(v)\mathrm{str}\rho(\alpha(y))[\alpha(u),[x,z]]
\\=&(-1)^{(|u|+|z|)(|x|+|y|)}\mathrm{str}\rho(x)\mathrm{str}\rho(\alpha(z))[\alpha(u),[y,v]]
\\&+(-1)^{|z|(|x|+|y|)}\mathrm{str}\rho(x)\mathrm{str}\rho(\alpha(z))[[y,u],\alpha(v)]
\\&-(-1)^{(|u|+|z|)(|x|+|y|)+|x||y|}\mathrm{str}\rho(y)\mathrm{str}\rho(\alpha(z))[\alpha(u),[x,v]]
\\&-(-1)^{|z|(|x|+|y|)+|x||y|}\mathrm{str}\rho(y)\mathrm{str}\rho(\alpha(z))[[x,u],\alpha(v)]
\\&-(-1)^{|u||z|+(|x|+|y|)(|z|+|u|)}\mathrm{str}\rho(x)\mathrm{str}\rho(\alpha(u))[\alpha(z),[y,v]]
\\&-(-1)^{|u|(|z|+|y|)}\mathrm{str}\rho(x)\mathrm{str}\rho(\alpha(u))[[y,z],\alpha(v)]
\\&+(-1)^{|u||z|+|x||y|+(|x|+|y|)(|z|+|u|)}\mathrm{str}\rho(y)\mathrm{str}\rho(\alpha(u))[\alpha(z),[x,v]]
\\&+(-1)^{|x||y|+|u|(|z|+|x|)}\mathrm{str}\rho(y)\mathrm{str}\rho(\alpha(u))[[x,z],\alpha(v)]
\\&+(-1)^{|z|(|x|+|y|)+|v|(|z|+|y|+|u|)}
\mathrm{str}\rho(x)\mathrm{str}\rho(\alpha(v))[\alpha(z),[y,u]]
\\&+(-1)^{|v|(|z|+|y|+|u|)}
\mathrm{str}\rho(x)\mathrm{str}\rho(\alpha(v))[[y,z],\alpha(u)]
\\&-(-1)^{|z|(|x|+|y|)+|v|(|z|+|x|+|u|)+|x||y|}
\mathrm{str}\rho(y)\mathrm{str}\rho(\alpha(v))[\alpha(z),[x,u]]
\\&-(-1)^{|v|(|z|+|u|+|x|)+|x||y|}
\mathrm{str}\rho(y)\mathrm{str}\rho(\alpha(v))[[x,z],\alpha(u)]
\end{align*}
Similarly, one gets
\begin{align*}R=&[[x,y,z]_{\rho},\alpha(u),\alpha(v)]_{\rho}
+(-1)^{|z|(|x|+|y|)}[\alpha(z),[x,y,u]_{\rho},\alpha(v)]_{\rho}
\\&+(-1)^{(|z|+|u|)(|x|+|y|)}[\alpha(z),\alpha(u),[x,y,v]_{\rho}]_{\rho}
\\=&-(-1)^{|u|(|y|+|z|)}\mathrm{str}\rho(x)\mathrm{str}\rho(\alpha(u))[[y,z],\alpha(v)]
\\&+(-1)^{|v|(|y|+|z|+|u|)}\mathrm{str}\rho(x)\mathrm{str}\rho(\alpha(v))[[y,z],\alpha(u)]
\\&+(-1)^{|x||y|+|u|(|x|+|z|)}\mathrm{str}\rho(y)\mathrm{str}\rho(\alpha(u))[[x,z],\alpha(v)]
\\&-(-1)^{|x||y|+|v|(|u|+|z|+|x|)}\mathrm{str}\rho(y)\mathrm{str}\rho(\alpha(v))[[x,z],\alpha(u)]
\\&-(-1)^{(|x|+|y|)(|z|+|u|)}\mathrm{str}\rho(z)\mathrm{str}\rho(\alpha(u))[[x,y],\alpha(v)]
\\&+(-1)^{(|z|+|v|)(|x|+|y|)+|v||u|}
\mathrm{str}\rho(z)\mathrm{str}\rho(\alpha(v))[[x,y],\alpha(u)]
\\&+(-1)^{|z|(|x|+|y|)}\mathrm{str}\rho(x)\mathrm{str}\rho(\alpha(z))[[y,u],\alpha(v)]
\\&+(-1)^{|z|(|x|+|y|)+|v|(|z|+|y|+|u|)}
\mathrm{str}\rho(x)\mathrm{str}\rho(\alpha(v))[\alpha(z),[y,u]]
\\&-(-1)^{(|z|(|x|+|y|)+|x||y|}
\mathrm{str}\rho(y)\mathrm{str}\rho(\alpha(z))[[x,u],\alpha(v)]
\\&-(-1)^{|z|(|x|+|y|)+|x||y|+|v|(|x|+|u|+|z|)}
\mathrm{str}\rho(y)\mathrm{str}\rho(\alpha(v))[\alpha(z),[x,u]]
\\&+(-1)^{(|y|+|x|)(|z|+|u|)}
\mathrm{str}\rho(u)\mathrm{str}\rho(\alpha(z))[[x,y],\alpha(v)]
\\&+(-1)^{(|z|+|u|)(|y|+|x|)+|v|(|x|+|z|+|y|)}
\mathrm{str}\rho(u)\mathrm{str}\rho(\alpha(v))[\alpha(z),[x,y]]
\\&+(-1)^{(|u|+|z|)(|x|+|y|)}\mathrm{str}\rho(x)\mathrm{str}\rho(\alpha(z))[\alpha(u),[y,v]]
\\&-(-1)^{(|u|+|z|)(|x|+|y|)+|u||z|}\mathrm{str}\rho(x)\mathrm{str}\rho(\alpha(u))[\alpha(z),[y,v]]
\\&-(-1)^{(|u|+|z|)(|x|+|y|)+|x||y|}\mathrm{str}\rho(y)\mathrm{str}\rho(\alpha(z))[\alpha(u),[x,v]]
\\&+(-1)^{(|u|+|z|)(|x|+|y|)+|x||y|+|u||z|}\mathrm{str}\rho(y)\mathrm{str}\rho(\alpha(u))[\alpha(z),[x,v]]
\\&+(-1)^{(|x|+|y|)(|z|+|u|+|v|)}\mathrm{str}\rho(v)\mathrm{str}\rho(\alpha(z))[\alpha(u),[x,y]]
\\&-(-1)^{(|x|+|y|)(|z|+|u|+|v|)+|u||z|}\mathrm{str}\rho(v)\mathrm{str}\rho(\alpha(u))[\alpha(z),[x,y]]
\\=&-(-1)^{|u|(|z|+|y|)}\mathrm{str}\rho(x)\mathrm{str}\rho(\alpha(u))[[y,z],\alpha(v)]
\\&+(-1)^{|v|(|z|+|y|+|u|)}\mathrm{str}\rho(x)\mathrm{str}\rho(\alpha(v))[[y,z],\alpha(u)]
\\&+(-1)^{|u|(|x|+|z|)+|x||y|}\mathrm{str}\rho(y)\mathrm{str}\rho(\alpha(u))[[x,z],\alpha(v)]
\\&-(-1)^{|v|(|z|+|x|+|u|)+|x||y|}\mathrm{str}\rho(y)\mathrm{str}\rho(\alpha(v))[[x,z],\alpha(u)]
\\&+(-1)^{|z|(|x|+|y|)}\mathrm{str}\rho(x)\mathrm{str}\rho(\alpha(z))[[y,u],\alpha(v)]
\\&+(-1)^{|z|(|x|+|y|)+|v|(|z|+|y|+|u|)}\mathrm{str}\rho(x)\mathrm{str}\rho(\alpha(v))[\alpha(z),[y,u]]
\\&-(-1)^{|x||y|+|z|(|x|+|y|)}\mathrm{str}\rho(y)\mathrm{str}\rho(\alpha(z))[[x,u],\alpha(v)]
\\&-(-1)^{|x||y|+|z|(|x|+|y|)+|v|(|z|+|x|+|u|)}\mathrm{str}\rho(y)\mathrm{str}\rho(\alpha(v))[\alpha(z),[x,u]]
\\&+(-1)^{(|x|+|y|)(|z|+|u|)}\mathrm{str}\rho(x)\mathrm{str}\rho(\alpha(z))[\alpha(u),[y,v]]
\\&-(-1)^{(|x|+|y|)(|z|+|u|)+|z||u|}
\mathrm{str}\rho(x)\mathrm{str}\rho(\alpha(u))[\alpha(z),[y,v]]
\\&-(-1)^{(|x|+|y|)(|z|+|u|)+|y||x|}
\mathrm{str}\rho(y)\mathrm{str}\rho(\alpha(z))[\alpha(u),[x,v]]
\\&+(-1)^{(|z|+|u|)(|x|+|y|)+|z||u|+|x||y|}
\mathrm{str}\rho(y)\mathrm{str}\rho(\alpha(u))[\alpha(z),[x,v]]
\end{align*}
By comparison, we have $L=R.$
By Proposition \ref{pro1.26} that if $\alpha$ is an
endomorphism of $(\mathfrak{g}, [\cdot,\cdot],\alpha),$  then $\alpha$ is an endomorphism of $(\mathfrak{g}, [\cdot,\cdot,\cdot]_{\rho},\alpha).$ That is, $(\mathfrak{g}, [\cdot,\cdot,\cdot]_{\rho},\alpha)$ is
multiplicative.
\end{proof}

\begin{proposition}
 Let $(\mathfrak{g}, [\cdot,\cdot])$ be a Lie superalgebra. $\alpha: V\rightarrow V$ be an algebra morphism.
 Then

 $\mathrm{(1)}$ $(\mathfrak{g}, [\cdot,\cdot]_{\alpha},\alpha)$ be a Hom-Lie superalgebra, where $[\cdot,\cdot]_{\alpha}=\alpha\circ [\cdot,\cdot].$

 $\mathrm{(2)}$ suppose that $\rho: \mathfrak{g}\rightarrow \mathrm{gl}(V)$ is a representation of $(\mathfrak{g}, [\cdot,\cdot]_{\alpha},\alpha),$ if $\mathrm{str}\rho\circ \alpha=\mathrm{str}\rho,$ then $[\cdot,\cdot,\cdot]_{\alpha,\rho}=[\cdot,\cdot,\cdot]_{\rho,\alpha}.$
\end{proposition}
\begin{proof} $\mathrm{(1)}$ It is clear.

 $\mathrm{(2)}$ Let $x_{1},x_{2},x_{3}\in \mathfrak{g},$ we have
\begin{align*}[x_{1},x_{2},x_{3}]_{\alpha,\rho}=&\mathrm{str}\rho(x_{1})[x_{2},x_{3}]_{\alpha}-(-1)^{|x_{1}||x_{2}|}
\mathrm{str}\rho(x_{2})[x_{1},x_{3}]_{\alpha}
\\&+(-1)^{|x_{3}|(|x_{1}|+|x_{2}|)}\mathrm{str}\rho(x_{3})[x_{1},x_{2}]_{\alpha}
\\=&\mathrm{str}\rho(x_{1})\alpha[x_{2},x_{3}]-(-1)^{|x_{1}||x_{2}|}
\mathrm{str}\rho(x_{2})\alpha[x_{1},x_{3}]
\\&+(-1)^{|x_{3}|(|x_{1}|+|x_{2}|)}\mathrm{str}\rho(x_{3})\alpha[x_{1},x_{2}]
\\=&\alpha\{\mathrm{str}\rho(x_{1})[x_{2},x_{3}]-(-1)^{|x_{1}||x_{2}|}
\mathrm{str}\rho(x_{2})[x_{1},x_{3}]
\\&+(-1)^{|x_{3}|(|x_{1}|+|x_{2}|)}\mathrm{str}\rho(x_{3})[x_{1},x_{2}]\}
\\=&\alpha([x_{1},x_{2},x_{3}]_{\rho})=[x_{1},x_{2},x_{3}]_{\rho,\alpha}.
\end{align*}
\end{proof}
\begin{definition}
Suppose that $(\mathfrak{g}, [\cdot,\cdot,\cdot],\alpha_{1},\alpha_{2})$ is a 3-ary-Hom-Lie superalgebra and $B$ a $\mathbb{Z}_{2}$-graded subspace
of $(\mathfrak{g}, [\cdot,\cdot,\cdot],\alpha_{1},\alpha_{2}),$ $B$ is called as a Hom-subalgebra of $(\mathfrak{g}, [\cdot,\cdot,\cdot],\alpha_{1},\alpha_{2}),$
if $\alpha_{i}(B)\subseteq B(i=1,2)$ and $[x_{1},x_{2},x_{3}]\in B, \forall x_{1},x_{2},x_{3}\in B;$
$B$ is called as a Hom-ideal of $(\mathfrak{g}, [\cdot,\cdot,\cdot],$ $\alpha_{1},\alpha_{2}),$
if $\alpha_{i}(B)\subseteq B(i=1,2)$ and $[x_{1},x_{2},x_{3}]\in B, \forall x_{1}\in B ,x_{2},x_{3}\in \mathfrak{g}.$
\end{definition}
Now we give two results about subalgebras and ideals of 3-ary-Hom-Lie superalgebras induced
by Hom-Lie superalgebras.
\begin{proposition} Let $(\mathfrak{g}, [\cdot,\cdot],\alpha_{1})$ be a Hom-Lie superalgebra, $\alpha_{2}: \mathfrak{g}\rightarrow \mathfrak{g}$ be an even linear map. $\rho: \mathfrak{g}\rightarrow \mathrm{gl}(V)$ is a reprensentation of $(\mathfrak{g}, [\cdot,\cdot],\alpha_{1})$ and $(\mathfrak{g}, [\cdot,\cdot,\cdot]_{\rho},\alpha_{1},\alpha_{2})$ is the induced 3-ary-Hom-Lie superalgebra.
Let $B$ be a Hom-subalgebra of $(\mathfrak{g}, [\cdot,\cdot],\alpha_{1}).$ If $\alpha_{2}(B)\subseteq B,$ then $B$ is also a Hom-subalgebra of
$\mathfrak{g}_{\rho}(=(\mathfrak{g}, [\cdot,\cdot,\cdot]_{\rho},\alpha_{1},\alpha_{2})).$
\end{proposition}
\begin{proof}  We have that $\alpha_{1}(B)\subseteq B,$ since $B$ is a subalgebra of $(\mathfrak{g}, [\cdot,\cdot],\alpha_{1}).$  Moreover, $\alpha_{2}(B)\subseteq B,$ then $\alpha_{i}(B)\subseteq B(i=1,2).$
Now let $x_{1},x_{2},x_{3}\in B,$ one gets
\begin{align*}[x_{1},x_{2},x_{3}]_{\rho}=&\mathrm{str}\rho(x_{1})[x_{2},x_{3}]-(-1)^{|x_{1}||x_{2}|}
\mathrm{str}\rho(x_{2})[x_{1},x_{3}]
+(-1)^{|x_{3}|(|x_{1}|+|x_{2}|)}\mathrm{str}\rho(x_{3})[x_{1},x_{2}],
\end{align*}
which is a linear combination of elements of $B$ and then belongs to $B$.
\end{proof}
\begin{proposition}Let $(\mathfrak{g}, [\cdot,\cdot],\alpha_{1})$ be a Hom-Lie superalgebra.
 $\rho: \mathfrak{g}\rightarrow \mathrm{gl}(V)$ is a representation of $(\mathfrak{g}, [\cdot,\cdot],\alpha_{1}),$
 $(\mathfrak{g}, [\cdot,\cdot,\cdot],\alpha_{1},\alpha_{2})$ is a the induced 3-ary-Hom-Lie superalgebra.
 Let $J$ be a Hom-ideal of $(\mathfrak{g}, [\cdot,\cdot],\alpha_{1})$ and
 If $\alpha_{2}(J)\subseteq J,$ then $J$ is a Hom-ideal of $\mathfrak{g}_{\rho}$ if and
only if
$[\mathfrak{g},\mathfrak{g}]\subseteq J$ or $J\subseteq \mathrm{ker}(\mathrm{str}\rho).$
\end{proposition}
\begin{proof}
 Let $J$ be an ideal of $\mathfrak{g}$ and $\alpha_{2}(J)\subseteq J,$ we have $\alpha_{i}(J)\subseteq J(i=1,2).$
Consider $y\in J$
and $x_{1},x_{2},x_{3}\in \mathfrak{g},$ one gets
\begin{align*}[x_{1},x_{2},y]_{\rho}=&\mathrm{str}\rho(x_{1})[x_{2},y]-(-1)^{|x_{1}||x_{2}|}
\mathrm{str}\rho(x_{2})[x_{1},y]
+(-1)^{|y|(|x_{1}|+|x_{2}|)}\mathrm{str}\rho(y)[x_{1},x_{2}].
\end{align*}
Since $[x_{2},y]$ and $[x_{1},y]$ belong to $J,$
 then, to obtain $[x_{1},x_{2},y]_{\rho}\in J$ it is
necessary and sufficient to have $\mathrm{str}\rho(y)[x_{1},x_{2}]\in J,$ which is equivalent to $\mathrm{str}\rho(y)=0$ or
$[x_{1},x_{2}]\in J.$
\end{proof}
\section{Solvability and nilpotency of $3$-ary multiplicative Hom-Lie superalgebras induced by Hom-Lie superalgebras}

\subsection{Solvability and nilpotency of $3$-ary multiplicative Hom-Lie superalgebras}
Now, we define the derived series, central descending series and the center of an $3$-ary-Hom-Lie
superalgebra, these generalization are relevant only in the case of multiplicative superalgebras.
\begin{definition}  Let $(\mathfrak{g}, [\cdot,\cdot,\cdot],\alpha)$ be a multiplicative 3-ary-Hom-Lie superalgebra, and $I$ a Hom-ideal of
$\mathfrak{g}.$ We define $D^{r}(I), r\in \mathbb{N},$ the derived series of $I$ by:
$$D^{0}(I)=I\ \mbox{and} \ D^{r+1}(I)=[D^{r}(I),D^{r}(I),D^{r}(I)].$$
\end{definition}

\begin{proposition} The subspaces $D^{r}(I)(r\in \mathbb{N})$ are subalgebras of $(\mathfrak{g}, [\cdot,\cdot,\cdot],\alpha).$
\end{proposition}
\begin{proof} We proceed by induction over $r\in \mathbb{N},$ the case of $r=0$ is trivial. Now suppose that
$D^{r}(I)$ is a Hom-subalgebra of $\mathfrak{g},$ we prove that $D^{r+1}(I)$ is a Hom-subalgebra of $\mathfrak{g}.$
Let $y\in D^{r+1}(I).$ Then we have
$$\alpha(y)=\alpha([y_{1} ,y_{2},y_{3}])=[\alpha(y_{1}),\alpha(y_{1}),\alpha(y_{1})], y_{1} ,y_{2},y_{3}\in D^{r}(I),$$
which is in $D^{r+1}(I),$ since $\alpha(y_{1}),\alpha(y_{1}),\alpha(y_{1})\in D^{r}(I).$ That is $ \alpha(D^{r+1}(I))\subseteq D^{r+1}(I).$
Let $x_{1} ,x_{2},x_{3}\in D^{r+1}(I),$ then
$$[x_{1} ,x_{2},x_{3}]=[[x_{11} ,x_{12},x_{13}],[x_{21} ,x_{22},x_{23}],[x_{31} ,x_{32},x_{33}]], x_{ij}\in D^{r}(I), i, j=1,2,3.$$
Hence $[x_{1} ,x_{2},x_{3}]\in D^{r+1}(I).$
\end{proof}
\begin{proposition}Let $(\mathfrak{g}, [\cdot,\cdot,\cdot],\alpha)$ be a multiplicative 3-ary-Hom-Lie superalgebra, and $I$ a Hom-ideal of
$\mathfrak{g}.$ If $\alpha$ is surjective, then $D^{r}(I)(r\in \mathbb{N})$ are Hom-ideals of $\mathfrak{g}.$
\end{proposition}
\begin{proof} Proof. We already have that $D^{r}(I)(r\in \mathbb{N})$ are subalgebras, we only need to prove that for
all $x_{1},x_{2}\in \mathfrak{g},$  and $y\in D^{r}(I), [x_{1},x_{2},y]\in D^{r}(I).$

We proceed by induction over $r\in \mathbb{N},$ the case of $r=0$ is trivial. Now suppose that $D^{r}(I)$
is an ideal of $\mathfrak{g},$ we prove that $D^{r+1}(I)$ is an ideal of $\mathfrak{g}.$
Let $x_{1},x_{2}\in \mathfrak{g}$ and $y\in D^{r+1}(I),$ we get
\begin{align*}
&[x_{1},x_{2},y]=[x_{1},x_{2},[y_{1},y_{2},y_{3}]], y_{1},y_{2},y_{3}\in D^{r}(I)
\\=&[\alpha(v_{1}),\alpha(v_{2}),[y_{1},y_{2},y_{3}]] \ \mbox{for some}\ v_{1},v_{2}\in \mathfrak{g}
\\=&[[v_{1},v_{2},y_{1}],\alpha(y_{2}),\alpha(y_{3})]+(-1)^{|y_{1}|(|v_{1}|+|v_{2}|)}[\alpha(y_{1}),[v_{1},v_{2},y_{2}],\alpha(y_{3})]
\\&+(-1)^{(|y_{1}|+|y_{1}|)(|v_{1}|+|v_{2}|)}[\alpha(y_{1}),\alpha(y_{2}),[v_{1},v_{2},y_{3}]].
\end{align*}
Hence $[x_{1} ,x_{2},y]\in D^{r+1}(I),$ since all the $\alpha(y_{i})\in D^{r}(I)$ and all the $[v_{1},v_{2},y_{i}]\in D^{r}(I)$($D^{r}(I)$ is a Hom-ideal),
and all the $[[v_{1},v_{2},y_{1}],\alpha(y_{2}),\alpha(y_{3})], [\alpha(y_{1}),[v_{1},v_{2},y_{2}],\alpha(y_{3})]$ and $[\alpha(y_{1}),$\\$\alpha(y_{2}),[v_{1},v_{2},y_{3}]]$ are in $D^{r+1}(I).$
\end{proof}
\begin{definition}  Let $(\mathfrak{g}, [\cdot,\cdot,\cdot],\alpha)$ be a multiplicative 3-ary-Hom-Lie superalgebra, and $I$ a Hom-ideal of
$\mathfrak{g}.$  We define $C^{r}(I)(r\in \mathbb{N})$ the central descending series of $I$ by
$C^{0}(I)=I$ and $C^{r+1}(I)=[C^{r}(I),I,I].$
\end{definition}

\begin{proposition} Let $(\mathfrak{g}, [\cdot,\cdot,\cdot],\alpha)$ be a multiplicative 3-ary-Hom-Lie superalgebra, and $I$ a Hom-ideal of
$\mathfrak{g}.$ If $\alpha$ is surjective, then $C^{r}(I)(r\in \mathbb{N})$ are Hom-ideals of $\mathfrak{g}.$
\end{proposition}
\begin{proof} We proceed by induction over $r\in \mathbb{N},$ the case of $r=0$ is trivial. Now suppose that
$C^{r}(I)$ is a Hom-ideal of $\mathfrak{g},$ we prove that $C^{r+1}(I)$ is a Hom-ideal of $\mathfrak{g},$
Let $y\in C^{r+1}(I).$ Then one gets
$$\alpha(y)=\alpha([y_{1},y_{2},\omega])=[\alpha(y_{1}),\alpha(y_{2}),\alpha(\omega)], y_{1},y_{2}\in I,\omega\in C^{r}(I),$$
which is in $C^{r+1}(I),$ since $\alpha(y_{2}),\alpha(\omega)\in I$ and $\alpha(\omega)\in C^{r}(I).$
That is $\alpha(C^{r+1}(I))\subseteq C^{r+1}(I).$
Let $x_{1},x_{2}\in \mathfrak{g}$ and
$y\in C^{r+1}(I).$
\begin{align*}
&[x_{1},x_{2},y]=[x_{1},x_{2},[y_{1},y_{2},\omega]], y_{1},y_{2}\in I,\omega\in C^{r}(I)
\\=&[\alpha(v_{1}),\alpha(v_{2}),[y_{1},y_{2},\omega]] \ \mbox{for some}\ v_{1},v_{2}\in \mathfrak{g}
\\=&[[v_{1},v_{2},y_{1}],\alpha(y_{2}),\alpha(\omega)]+(-1)^{|y_{1}|(|v_{1}|+|v_{2}|)}[\alpha(y_{1}),[v_{1},v_{2},y_{2}],\alpha(\omega)]
\\&+(-1)^{(|y_{1}|+|y_{1}|)(|v_{1}|+|v_{2}|)}[\alpha(y_{1}),\alpha(y_{2}),[v_{1},v_{2},\omega]],
\end{align*}
which is in $C^{r+1}(I),$ since all the $\alpha(y_{i})\in I, \alpha(\omega)\in C^{r}(I)$ and all the $[v_{1},v_{2},y_{i}]\in I$($I$ is a Hom-ideal),
and $[v_{1},v_{2},\omega]\in C^{r}(I)(C^{r}(I)$ is a Hom-ideal). Therefore, all the $[[v_{1},v_{2},y_{1}],\alpha(y_{2}),\alpha(\omega)],[\alpha(y_{1}),[v_{1},v_{2},y_{2}],\alpha(\omega)]$ and $[\alpha(y_{1}),\alpha(y_{2}),[v_{1},v_{2},\omega]]$ are in $C^{r+1}(I).$
\end{proof}
\begin{definition}  Let $(\mathfrak{g}, [\cdot,\cdot,\cdot],\alpha)$ be a multiplicative 3-ary-Hom-Lie superalgebra, and $I$ an a Hom-ideal of
$\mathfrak{g}.$ The Hom-ideal $I$ is said to be solvable if there exists $r\in \mathbb{N}$ such that  $D^{r}(I)=\{0\}.$ It is said to
be nilpotent if there exists $r\in \mathbb{N}$ such that  $C^{r}(I)=\{0\}.$
\end{definition}
\begin{definition} Let $(\mathfrak{g}, [\cdot,\cdot,\cdot],\alpha)$ be a multiplicative 3-ary-Hom-Lie superalgebra. Define the center
of $\mathfrak{g},$ denoted by $Z(\mathfrak{g}),$ as:
$Z(\mathfrak{g})=\{z\in\mathfrak{g}\ | \ [x_{1},x_{2},z]=0,\forall x_{1},x_{2}\in \mathfrak{g}\}.$
\end{definition}
\begin{proposition}
 Let $(\mathfrak{g}, [\cdot,\cdot,\cdot],\alpha)$ be a multiplicative 3-ary-Hom-Lie superalgebra. If $\alpha$ is surjective,
then the center of $\mathfrak{g}$ is a Hom-ideal of $\mathfrak{g}.$
\end{proposition}
\begin{proof}  Let $z\in Z(\mathfrak{g})$ and $x_{1},x_{2}\in \mathfrak{g},$ we put $x_{i}=\alpha(u_{i}),$ for $i=1,2,$ then we
have
$$[x_{1},x_{2}, \alpha(z)]=[\alpha(u_{1}),\alpha(u_{2}),\alpha(z)]=\alpha([u_{1},u_{2},z])=0,$$
that is $\alpha(Z(\mathfrak{g}))\subseteq Z(\mathfrak{g}).$
Let $z\in Z(\mathfrak{g})$ and $x_{1},x_{2},y_{1},y_{2}\in \mathfrak{g},$ we have
$$[x_{1},x_{2},[y_{1},y_{2},z]]=[x_{1},x_{2},0] = 0,$$
which means that $Z(\mathfrak{g})$ is a Hom-ideal of $\mathfrak{g}.$
\end{proof}

\subsection{Solvability and nilpotency of $3$-ary multiplicative Hom-Lie superalgebras induced by Hom-Lie superalgebras}

Now we show the relationships between central descending series, derived series and center
of a 3-ary Hom-Lie superalgebra, and those of the induced 3-ary-Hom-Lie superalgebra.

\begin{theorem} \label{th4.9} Let $(\mathfrak{g}, [\cdot,\cdot],\alpha)$ be a Hom-Lie superalgebra. $\rho: \mathfrak{g}\rightarrow \mathrm{gl}(V)$ is a reprensentation of $(\mathfrak{g}, [\cdot,\cdot],\alpha),$ $\beta: \mathfrak{g}\rightarrow \mathfrak{g}$
a linear map satisfying the conditions of Theorem \ref{th1.24} and $(\mathfrak{g}, [\cdot,\cdot,\cdot]_{\rho},\alpha,\beta)$ the induced
3-ary Hom-Lie superalgebra. Then the induced algebra is solvable, more precisely $D^{2}(\mathfrak{g}_{\rho})=0,$ i.e.
$(D^{1}(\mathfrak{g}_{\rho})=[\mathfrak{g},\mathfrak{g},\mathfrak{g}]_{\rho},[\cdot,\cdot,\cdot]_{\rho})$
is abelian.
\end{theorem}
\begin{proof}  Let $x_{1},x_{2},x_{3}\in [\mathfrak{g},\mathfrak{g},\mathfrak{g}]_{\rho}, x_{i}=[x^{1}_{i},x^{2}_{i},x^{3}_{i}]_{\rho},
\forall 1\leq i\leq 3,$ then
\begin{align*}
&[x_{1},x_{2},x_{3}]_{\rho}=[[x^{1}_{1},x^{2}_{1},x^{3}_{1}]_{\rho},  [x^{1}_{2},x^{2}_{2},x^{3}_{2}]_{\rho},[x^{1}_{3},x^{2}_{3},x^{3}_{3}]_{\rho}]_{\rho}
\\=&\mathrm{str}\rho([x^{1}_{1},x^{2}_{1},x^{3}_{1}]_{\rho})[[x^{1}_{2},x^{2}_{2},x^{3}_{2}]_{\rho},[x^{1}_{3},x^{2}_{3},x^{3}_{3}]_{\rho}]
\\&-(-1)^{|x_{1}||x_{2}|}
\mathrm{str}\rho([x^{1}_{2},x^{2}_{2},x^{3}_{2}]_{\rho})[[x^{1}_{1},x^{2}_{1},x^{3}_{1}]_{\rho},[x^{1}_{3},x^{2}_{3},x^{3}_{3}]_{\rho}]
\\&+(-1)^{|x_{3}|(|x_{1}|+|x_{2}|)}\mathrm{str}\rho([x^{1}_{3},x^{2}_{3},x^{3}_{3}]_{\rho}])[[x^{1}_{1},x^{2}_{1},x^{3}_{1}]_{\rho},  [x^{1}_{2},x^{2}_{2},x^{3}_{2}]_{\rho}]
\\=&0,
\end{align*}
since $\mathrm{str}\rho([\cdot,\cdot,\cdot]_{\rho})=0.$
\end{proof}
\begin{proposition} Let $(\mathfrak{g}, [\cdot,\cdot],\alpha)$ be a multiplicative Hom-Lie superalgebra
satisfying $\mathrm{str}\rho\circ \alpha=\mathrm{str}\rho$ and $(\mathfrak{g}, [\cdot,\cdot,\cdot]_{\rho},\alpha)$ the induced
3-ary-Hom-Lie superalgebra. Let $c\in Z(\mathfrak{g}),$ if
$\mathrm{str}\rho(c)=0,$ then $c\in Z(\mathfrak{g}_{\rho}).$ Moreover, if $\mathfrak{g}$ is not abelian, then $\mathrm{str}\rho(c)=0$ if and only if $c\in Z(\mathfrak{g}_{\rho}).$
\end{proposition}
\begin{proof} Let $c\in Z(\mathfrak{g})$ and $x_{1},x_{2}\in \mathfrak{g},$
\begin{align*}&[x_{1},x_{2},c]_{\rho}
\\=&\mathrm{str}\rho(x_{1})[x_{2},c]-(-1)^{|x_{1}||x_{2}|}
\mathrm{str}\rho(x_{2})[x_{1},c]+(-1)^{|c|(|x_{1}|+|x_{2}|)}\mathrm{str}\rho(c)[x_{1}, x_{2}]
\\=&(-1)^{|c|(|x_{1}|+|x_{2}|)}\mathrm{str}\rho(c)[x_{1}, x_{2}].
\end{align*}
If $\mathrm{str}\rho(c)=0,$ then $c\in Z(\mathfrak{g}_{\rho}).$
Conversely, if $c\in Z(\mathfrak{g}_{\rho})$ and $\mathfrak{g}$ is not abelian, then $\mathrm{str}\rho(c)=0.$
\end{proof}
\begin{proposition} Let $(\mathfrak{g}, [\cdot,\cdot],\alpha)$ be a non-abelian multiplicative Hom-Lie superalgebra
satisfying $\mathrm{str}\rho\circ \alpha=\mathrm{str}\rho$ and $(\mathfrak{g}, [\cdot,\cdot,\cdot]_{\rho},\alpha)$ the induced
3-ary-Hom-Lie superalgebra. If
$\mathrm{str}\rho(Z(\mathfrak{g}))\neq \{0\},$ then $\mathfrak{g}_{\rho}$ is not abelian.
\end{proposition}
\begin{proof} Let $x_{1},x_{2}\in \mathfrak{g}$ such that $[x_{1},x_{2}]\neq 0$ and $c\in Z(\mathfrak{g})$ such that $\mathrm{str}\rho(c)\neq 0,$ then we
have
\begin{align*}&[x_{1},x_{2},c]_{\rho}
\\=&\mathrm{str}\rho(x_{1})[x_{2},c]-(-1)^{|x_{1}||x_{2}|}
\mathrm{str}\rho(x_{2})[x_{1},c]+(-1)^{|c|(|x_{1}|+|x_{2}|)}\mathrm{str}\rho(c)[x_{1}, x_{2}]
\\=&(-1)^{|c|(|x_{1}|+|x_{2}|)}\mathrm{str}\rho(c)[x_{1}, x_{2}],
\end{align*}
which means that $\mathfrak{g}_{\rho}$ is not abelian.
\end{proof}
\begin{proposition} Let $(\mathfrak{g}, [\cdot,\cdot],\alpha)$ be a multiplicative Hom-Lie superalgebra
satisfying $\mathrm{str}\rho\circ \alpha=\mathrm{str}\rho$ and $(\mathfrak{g}, [\cdot,\cdot,\cdot]_{\rho},\alpha)$ the induced
3-ary-Hom-Lie superalgebra. Let $(C^{p}(\mathfrak{g}))_{p}$ be the central
descending series of $\mathfrak{g},$ and $(C^{p}(\mathfrak{g}_{\rho}))_{p}$ be the central descending series of $\mathfrak{g}_{\rho}.$  Then we have
 $C^{p}(\mathfrak{g}_{\rho})\subset C^{p}(\mathfrak{g}), \forall p\in\mathbb{N}.$
If there exists $u\in \mathfrak{g}$ such that $[u,x_{1},x_{2}]_{\rho}=[x_{1},x_{2}], \forall x_{1},x_{2}\in \mathfrak{g},$ then
$C^{p}(\mathfrak{g}_{\rho})=C^{p}(\mathfrak{g}), \forall p\in\mathbb{N}.$
\end{proposition}
\begin{proof} Theorem \ref{th1.27} provides that $\mathfrak{g}_{\rho}$ is multiplicative. We proceed by induction over $p\in\mathbb{N}.$
The case of $p=0$ is trivial, for $p=1,$
$ \forall x=[x_{1},x_{2},x_{3}]_{\rho}\in C^{1}(\mathfrak{g}_{\rho}),$
we have
\begin{align*}x=&\mathrm{str}\rho(x_{1})[x_{2},x_{3}]-(-1)^{|x_{1}||x_{2}|}
\mathrm{str}\rho(x_{2})[x_{1},x_{3}]
\\&+(-1)^{|x_{3}|(|x_{1}|+|x_{2}|)}\mathrm{str}\rho(x_{3})[x_{1}, x_{2}],
\end{align*}
which is a linear combination of elements of $C^{1}(\mathfrak{g})$ and then is an element of $C^{1}(\mathfrak{g}).$ Suppose
now that there exists  $u\in \mathfrak{g}$ such that $[u,x_{1},x_{2}]_{\rho}=[x_{1},x_{2}], \forall x_{1},x_{2}\in \mathfrak{g}.$ Then for
$x=[x_{1},x_{2}]\in C^{1}(\mathfrak{g}), x=[u,x_{1},x_{2}]_{\rho}$ and hence it is an element of $C^{1}(\mathfrak{g})_{\rho}.$

Now, we suppose this proposition true for some $p\in\mathbb{N},$ and let $x\in C^{p+1}(\mathfrak{g}_{\rho}).$ Then
$x=[a,x_{1},x_{2}]_{\rho}$ with $x_{1},x_{2}\in \mathfrak{g}$ and $a\in C^{p}(\mathfrak{g}_{\rho}),$
\begin{align*}x=&[a,x_{1},x_{2}]_{\rho}=-(-1)^{|x_{1}||a|}
\mathrm{str}\rho(x_{1})[a,x_{3}]+(-1)^{|x_{2}|(|a|+|x_{1}|)}\mathrm{str}\rho(x_{2})[a, x_{1}],
\end{align*}
which is an element of $C^{p+1}(\mathfrak{g})$ because $a\in C^{p}(\mathfrak{g}_{\rho})\subset C^{p}(\mathfrak{g}).$ Assume there exists $u \in \mathfrak{g}$
such that $[u,x_{1},x_{2}]_{\rho}=[x_{1},x_{2}], \forall x_{1},x_{2}\in \mathfrak{g}.$ Then if $x\in C^{p+1}(\mathfrak{g}),$  we have $x=
[a,x_{1}]$  with $a \in C^{p}(\mathfrak{g})$ and $ x_{1}\in \mathfrak{g}.$ Therefore
$x=[a,x_{1}]=[u,a,x_{1}]_{\rho}=(-1)^{|u|(|a|+|x_{1}|)}[a,x_{1},u]_{\rho}\in C^{p+1}(\mathfrak{g}_{\rho}).$
\end{proof}
\begin{remark}\label{re4.13}
 It also results from the preceding proposition that
$$D^{1}(\mathfrak{g}_{\rho})=[\mathfrak{g},\mathfrak{g},\mathfrak{g}]_{\rho}\subset D^{1}(\mathfrak{g})=[\mathfrak{g},\mathfrak{g}],$$
and if there exists $u \in \mathfrak{g}$ such that
$[u,x_{1},x_{2}]_{\rho}=[x_{1},x_{2}], \forall x_{1},x_{2}\in \mathfrak{g}.$
Then $D^{1}(\mathfrak{g}_{\rho})=D^{1}(\mathfrak{g}).$ For the rest of the derived series, we have obviously the first inclusion
by Theorem \ref{th4.9}, which states also that all induced algebras are solvable.
\end{remark}
\begin{theorem}
 Let $(\mathfrak{g}, [\cdot,\cdot],\alpha)$ be a  multiplicative Hom-Lie superalgebra
satisfying $\mathrm{str}\rho\circ \alpha=\mathrm{str}\rho$ and $(\mathfrak{g}, [\cdot,\cdot,\cdot]_{\rho},\alpha)$ the induced
3-ary-Hom-Lie superalgebra. Then, if $\mathfrak{g}$ is nilpotent of class
$p,$ we have $\mathfrak{g}_{\rho}$ is nilpotent of class at most $p.$ Moreover, if there exists $u \in \mathfrak{g}$ such that
$[u,x_{1},x_{2}]_{\rho}=[x_{1},x_{2}], \forall x_{1},x_{2}\in \mathfrak{g},$ then $\mathfrak{g}$ is nilpotent of class $p$ if and only if $\mathfrak{g}_{\rho}$ is
nilpotent of class $p.$
\end{theorem}
\begin{proof}
 Theorem \ref{th1.27} provides that $\mathfrak{g}_{\rho}$ is multiplicative.

(1) Suppose that $(\mathfrak{g},[\cdot,\cdot])$ is nilpotent of class $p\in\mathbb{N},$ then $C^{P}(\mathfrak{g})=\{0\}.$  By the preceding
proposition, $C^{P}(\mathfrak{g}_{\rho})\subseteq C^{P}(\mathfrak{g})=\{0\},$ therefore $(\mathfrak{g},[\cdot,\cdot,\cdot]_{\rho})$ is nilpotent of class at most
$p.$

(2) We suppose now that $(\mathfrak{g},[\cdot,\cdot,\cdot]_{\rho})$ is nilpotent of class $p\in\mathbb{N},$ and that there exists $u \in \mathfrak{g}$
such that $[u,x_{1},x_{2}]_{\rho}=[x_{1},x_{2}], \forall x_{1},x_{2}\in \mathfrak{g},$ then $C^{p}(\mathfrak{g}_{\rho})=0.$ By the
preceding proposition, $C^{p}(\mathfrak{g})=C^{p}(\mathfrak{g}_{\rho})=\{0\}.$ Therefore $(\mathfrak{g},[\cdot,\cdot])$ is nilpotent, since
$C^{p-1}(\mathfrak{g})=C^{p-1}(\mathfrak{g}_{\rho})\neq\{0\},$ hence $(\mathfrak{g},[\cdot,\cdot,\cdot]_{\rho})$ and $(\mathfrak{g},[\cdot,\cdot])$ have the same nilpotency class.
\end{proof}

\section{Central extensions of $3$-ary multiplicative Hom-Lie superalgebras induced by Hom-Lie superalgebras}
\begin{definition}
Let $(\mathfrak{g}, [\cdot,\cdot],\alpha)$ be a multiplicative Hom-Lie superalgebra. We call central extensions of $\mathfrak{g}$ the space $\bar{\mathfrak{g}}=\mathfrak{g}\oplus Kc$ equipped with the bracket $[\cdot,\cdot]_{c}$ and the morphism $\alpha_{c}$ defined by:
$$[x_{1},x_{2}]_{c}=[x_{1},x_{2}]+\omega(x_{1},x_{2})c \ \ \text{and} \ \ [x_{1},c]=0, \forall x_{1},x_{2}\in \mathfrak{g},$$
$$\bar{\alpha}(\bar{x})=\alpha(x)+(\lambda(\bar{x}))c, \forall \bar{x}=x+x_{c}c\in \bar{\mathfrak{g}}, x\in \mathfrak{g}.$$
where $\lambda:\bar{\mathfrak{g}}\rightarrow K$ is a linear map and $\omega:\mathfrak{g}\times \mathfrak{g}\rightarrow K$
is a skew-supersymmetric bilinear map such that $[\cdot,\cdot]_{c}$ and $\alpha_{c}$ satisfy the Hom-Jacobi identity.
\end{definition}

\begin{definition}\cite{lyc} Let $(\mathfrak{g}, [\cdot,\cdot],\alpha)$ be a multiplicative Hom-Lie superalgebra and $K$ be a field.  The set of $p$-cochains on $\mathfrak{g}$ with coefficients in $K$, which we denote by $C^{p}(\mathfrak{g};K)$, is the set of skew-supersymmetric $p$-linear maps from $\mathfrak{g}\times\cdots \times \mathfrak{g}$ ($p$-times) to $K$:\\
 $$C^{p}(\mathfrak{g};K) \triangleq \{f: \wedge ^{p}\mathfrak{g}\rightarrow \text{ K is a linear map}\}.$$

The set of $p$-Hom-cochains on $\mathfrak{g}$ with coefficients in $K$,  which we
denote by $C^{p}_{\alpha}(\mathfrak{g};K)$ is given by
$$C^{p}_{\alpha}(\mathfrak{g};K)=\{f\in C^{p}(\mathfrak{g};K)|\alpha \circ f=f\circ \alpha\}.$$

Associated to the $\alpha ^{0}$-adjoint representation,  the coboundary operator $d_s:C_{\alpha}^{p}(\mathfrak{g};K)\rightarrow C_{\alpha}^{p+1}(\mathfrak{g};K)$ is given by
\begin{eqnarray*}
d_{s}f(x_{0}, \cdots, x_{k})&=&\sum_{i<j}(-1)^{i+j}(-1)^{(|x_{0}|+\cdots +|x_{i-1}|)|x_{i}|}(-1)^{(|x_{0}|+\cdots +|x_{j-1}|)|x_{j}|}(-1)^{|x_{i}||x_{j}|}\\
&&f([x_{i}, x_{j}], \alpha(x_{0}), \cdots, \widehat{\alpha(x_{i})}, \cdots, \widehat{\alpha(x_{i})}, \cdots, \alpha(x_{k})).
\end{eqnarray*}
\end{definition}
The elements of $Z^{p}(\mathfrak{g};K)=\mathrm{ker}\,d^{p}$ are called $p$-cocycles, those of
$B^{p}(\mathfrak{g};K)=\mathrm{Im}\,d^{p}$ are called coboundaries. The quotient $H^{p}=\frac{Z^{p}}{B^{p}}$
is the $p$-th cohomology group. We sometimes add in subscript the representation used in
the cohomology complex, for example $Z^{p}_{ad}(\mathfrak{g},\mathfrak{g})$ denotes the set of $p$-cocycle for the adjoint
cohomology and $Z^{p}_{0}(\mathfrak{g},K)$ denotes the set of $p$-cocycle for the scalar cohomology.
\begin{proposition} $\mathrm{(1)}$ The bracket of a central extension of a Hom-Lie superalgebra
satisfies the Hom-Jacobi identity if and only if the map $\omega$ is a $2$-cocycle for the scalar
cohomology of Hom-Lie superalgebras.

$\mathrm{(2)}$ Two central extensions defined by two maps $\omega_{1}$ and $\omega_{2}$ are isomorphic if and only if
$\omega_{2}-\omega_{1}$ is a $2$-coboundary for the scalar cohomology of Hom-Lie superalgebras.
\end{proposition}
\begin{proof}
(1) Let $(\mathfrak{g}, [\cdot,\cdot],\alpha)$ be a multiplicative Hom-Lie superalgebra. Let $\omega$ be
a skew-supersymmetric bilinear form on $\bar{\mathfrak{g}}=\mathfrak{g}\oplus Kc$ the bracket $[\cdot,\cdot]_{c}$ by:
$$[x_{1},x_{2}]_{c}=[x_{1},x_{2}]+\omega(x_{1},x_{2})c \ \ \text{and} \ \ [x_{1},c]=0, \forall x_{1},x_{2}\in \mathfrak{g}.$$
Then we have
\begin{align*}
&(-1)^{|z||x|}[\alpha(x),[y,z]_{c}]_{c}+(-1)^{|z||y|}[\alpha(z),[x,y]_{c}]_{c}+(-1)^{|x||y|}[\alpha(y),[z,x]_{c}]_{c}
\\=&(-1)^{|z||x|}([\alpha(x),[y,z]]+\omega(\alpha(x),[y,z])c)+(-1)^{|z||y|}([\alpha(z),[x,y]]+\omega(\alpha(z),[x,y])c)
\\&+(-1)^{|x||y|}([\alpha(y),[z,x]]+\omega(\alpha(y),[z,x])c)
\\=&(-1)^{|z||x|}[\alpha(x),[y,z]]+(-1)^{|z||y|}[\alpha(z),[x,y]]
+(-1)^{|x||y|}[\alpha(y),[z,x]]
\\&+(-1)^{|z||x|}\omega(\alpha(x),[y,z])c+(-1)^{|z||y|}\omega(\alpha(z),[x,y])c
+(-1)^{|x||y|}\omega(\alpha(y),[z,x])c
\\=&(-1)^{|z||x|}\omega(\alpha(x),[y,z])c+(-1)^{|z||y|}\omega(\alpha(z),[x,y])c
+(-1)^{|x||y|}\omega(\alpha(y),[z,x])c
\\=&(-1)^{|x||z|}d^{2}\omega(x,y,z)c.
\end{align*}
That is, Hom-Jacobi identity is satisfied if and only if the map $\omega$ is a $2$-cocycle for the scalar
cohomology of $\mathfrak{g}.$

(2) Let $\omega_{1},\omega_{2}\in Z^{p}_{0}(\mathfrak{g},K)$ such that $\omega_{2}-\omega_{1}=\alpha([\cdot,\cdot])$ with $\alpha\in C^{1}_{\alpha}(\mathfrak{g},K).$ Let $(\bar{\mathfrak{g}},[\cdot,\cdot]_{\omega_{1}})$ and $(\bar{\mathfrak{g}},[\cdot,\cdot]_{\omega_{2}})$
 be two central extensions of $\mathfrak{g}$ defined by $\omega_{1}$ and $\omega_{2}$ respectively. Consider
 $$f: (\bar{\mathfrak{g}},[\cdot,\cdot]_{\omega_{1}})\rightarrow (\bar{\mathfrak{g}},[\cdot,\cdot]_{\omega_{2}})$$
defined by $f(x)=x+\alpha(p_{A}(x))c,$ where $p_{A}$ is the projection of range $\mathfrak{g}.$ We have
\begin{align*}
&f([x_{1},x_{2}]_{\omega_{1}})=[x_{1},x_{2}]_{\omega_{1}}+\alpha[x_{1},x_{2}]c
\\=&[x_{1},x_{2}]+\omega_{1}(x_{1},x_{2})c+\alpha[x_{1},x_{2}]c
\\=&[x_{1},x_{2}]+\omega_{2}(x_{1},x_{2})c=[x_{1},x_{2}]_{\omega_{2}}
\\=&[x_{1}+\alpha(p_{A}(x_{1}))c,x_{2}+\alpha(p_{A}(x_{2}))c]_{\omega_{2}}
\\=&[f(x_{1}),f(x_{2})]_{\omega_{2}}.
\end{align*}
That means $f$ is a Hom-Lie superalgebras homomorphism. Let us prove now that is an  isomorphism:
\begin{align*}
&\mathrm{ker}(f)=\{x\in \bar{\mathfrak{g}}|  f(x)=0\}
\\=&\{x\in \bar{\mathfrak{g}}\ |\ x+\alpha(p_{A}(x))c=0\}
\\=&\{x\in \bar{\mathfrak{g}}\ |\ p_{A}(x)+(x_{c}+\alpha(p_{A}(x)))c=0\}(x=p_{A}(x)+x_{c}c)
\\=&\{x\in \bar{\mathfrak{g}}\ |\ p_{A}(x)=0\ \ \mbox{and}\ \ x_{c}+\alpha(p_{A}(x))=0\}=\{0\},
\end{align*}
which means that $f$ is injective. Therefore one concludes, when $\mathfrak{g}$ is finite dimensional, that it is  bijective. We prove now that $f$
is surjective,  so the result holds in infinite dimensional case:
\begin{align*}
&\mathrm{Im}(f)=\{f(x) \ | \  x\in \bar{\mathfrak{g}}\}
\\=&\{ x+\alpha(p_{A}(x))c \ |\ x\in \bar{\mathfrak{g}}\}
\\=&\{p_{A}(x)+(x_{c}+\alpha(p_{A}(x)))c \ | \ x=p_{A}(x)+x_{c}c\in \bar{\mathfrak{g}}\}=\bar{\mathfrak{g}},
\end{align*}
which means that $f$ is an  isomorphism of Hom-Lie superalgebras.
\end{proof}

Now we show the relationship between the central extensions of a Hom-Lie superalgebra
and those of the induced 3-ary-Hom-Lie superalgebra (by some supertrace str):
\begin{theorem} Let $(\mathfrak{g}, [\cdot,\cdot],\alpha)$ be a multiplicative Hom-Lie superalgebra satisfying $\mathrm{str}\rho\circ\alpha=\mathrm{str}\rho,$ and $(\mathfrak{g}, [\cdot,\cdot,\cdot]_{\rho},\alpha)$ be the induced (multiplicative) 3-ary-Hom-Lie superalgebra. Let $(\bar{\mathfrak{g}}, [\cdot,\cdot]_{c}, \alpha_{c})$ be a central extension of $(\mathfrak{g}, [\cdot,\cdot],\alpha),$
where $$\bar{\mathfrak{g}}=\mathfrak{g}\oplus Kc,\, [x_{1},x_{2}]_{c}=[x_{1},x_{2}]+\omega(x_{1},x_{2})c,
\bar{\alpha}(\bar{x})=\alpha(x)+(\lambda(\bar{x}))c,$$
with $\lambda:\bar{\mathfrak{g}}\rightarrow K$, and assume that $\mathrm{str}\rho$ extends to $\bar{\mathfrak{g}}$ by $\mathrm{str}\rho(c)=0.$
Then 3-ary-Hom-Lie superalgebra $(\bar{\mathfrak{g}}, [\cdot,\cdot,\cdot]_{c,\rho}, \alpha_{c})$ induced by $(\bar{\mathfrak{g}}, [\cdot,\cdot]_{c}, \alpha_{c})$ is a central extension of $(\mathfrak{g}, [\cdot,\cdot,\cdot]_{\rho},\alpha)$,
where
$$[x_{1},x_{2},x_{3}]_{c,\rho}=[x_{1},x_{2},x_{3}]_{\rho}+\omega_{\rho}(x_{1},x_{2},x_{3})c$$
with \begin{align*}
\omega_{\rho}(x_{1},x_{2},x_{3})&=\mathrm{str}\rho(x_{1})\omega(x_{2},x_{3})-(-1)^{|x_{1}||x_{2}|}
\mathrm{str}\rho(x_{2})\omega(x_{1},x_{3})
\\&+(-1)^{|x_{3}|(|x_{1}|+|x_{2}|)}\mathrm{str}\rho(x_{3})\omega(x_{1},x_{2}).
\end{align*}
\end{theorem}
\begin{proof} We consider the algebra
$(\bar{\mathfrak{g}}, [\cdot,\cdot,\cdot]_{c,\rho}, \alpha_{c})$ induced by $(\bar{\mathfrak{g}}, [\cdot,\cdot]_{c}, \alpha_{c}).$
Let $x_{1},x_{2},x_{3}$\\$\in \mathfrak{g},$ then one gets
\begin{align*}
[x_{1},x_{2},x_{3}]_{c,\rho}&=\mathrm{str}\rho(x_{1})[x_{2},x_{3}]_{c}-(-1)^{|x_{1}||x_{2}|}
\mathrm{str}\rho(x_{2})[x_{1},x_{3}]_{c}
\\&+(-1)^{|x_{3}|(|x_{1}|+|x_{2}|)}\mathrm{str}\rho(x_{3})[x_{1},x_{2}]_{c}
\\=&\mathrm{str}\rho(x_{1})([x_{2},x_{3}]+\omega(x_{2},x_{3})c)-(-1)^{|x_{1}||x_{2}|}
\mathrm{str}\rho(x_{2})([x_{1},x_{3}]+\omega(x_{1},x_{3})c)
\\&+(-1)^{|x_{3}|(|x_{1}|+|x_{2}|)}\mathrm{str}\rho(x_{3})([x_{1},x_{2}]+\omega(x_{1},x_{2})c)
\\=&\mathrm{str}\rho(x_{1})[x_{2},x_{3}]-(-1)^{|x_{1}||x_{2}|}
\mathrm{str}\rho(x_{2})[x_{1},x_{3}]
+(-1)^{|x_{3}|(|x_{1}|+|x_{2}|)}\mathrm{str}\rho(x_{3})[x_{1},x_{2}]
\\&+\mathrm{str}\rho(x_{1})\omega(x_{2},x_{3})c-(-1)^{|x_{1}||x_{2}|}
\mathrm{str}\rho(x_{2})\omega(x_{1},x_{3})c
\\&+(-1)^{|x_{3}|(|x_{1}|+|x_{2}|)}\mathrm{str}\rho(x_{3})\omega(x_{1},x_{2})c
\\=&[x_{1},x_{2},x_{3}]_{\rho}+\omega_{\rho}(x_{1},x_{2},x_{3})c.
\end{align*}
The map  \begin{align*}
\omega_{\rho}(x_{1},x_{2},x_{3})&=\mathrm{str}\rho(x_{1})\omega(x_{2},x_{3})-(-1)^{|x_{1}||x_{2}|}
\mathrm{str}\rho(x_{2})\omega(x_{1},x_{3})
\\&+(-1)^{|x_{3}|(|x_{1}|+|x_{2}|)}\mathrm{str}\rho(x_{3})\omega(x_{1},x_{2})
\end{align*} is a skew-supersymmetric
trilinear form, and $[\cdot,\cdot,\cdot]_{c,\rho}$ satisfies the Hom-Nambu identity. We also have
$$
[x_{1},x_{2},c]_{c,\rho}=\mathrm{str}\rho(x_{1})[x_{2},c]_{c}-(-1)^{|x_{1}||x_{2}|}
\mathrm{str}\rho(x_{2})[x_{1},c]_{c}
+(-1)^{|c|(|x_{1}|+|x_{2}|)}\mathrm{str}\rho(c)[x_{1},x_{2}]_{c}
=0,
$$
since $[x_{1},c]_{c}=[x_{2},c]_{c}=0$ and $\mathrm{str}\rho(c)=0.$
Therefore $(\bar{\mathfrak{g}}, [\cdot,\cdot,\cdot]_{c,\rho}, \alpha_{c})$ is a central extension of $(\mathfrak{g}, [\cdot,\cdot,\cdot]_{\rho},\alpha).$
\end{proof}
\section{Cohomology for $3$-ary multiplicative Hom-Lie superalgebras induced by Hom-Lie superalgebras}

\begin{definition} \cite{gbl} \label{def:delta}For $m\geq 1,$ we call $m$-coboundary operator of the $n$-ary multiplicative Hom-Nambu-Lie superalgebra $(\mathfrak{g}, [\cdot,\cdots,\cdot],\alpha)$
the even linear map $\delta^{m}: C^m(\mathfrak{g}, V)\rightarrow C^{m+1}(\mathfrak{g}, V)$ by
\begin{align*}
&(\delta^{m} f)(\mathscr{X}_1,\cdots,\mathscr{X}_m, \mathscr{X}_{m+1}, z)\\
=&\sum_{i<j}(-1)^i(-1)^{|\mathscr{X}_i|(|\mathscr{X}_{i+1}|+\cdots+|\mathscr{X}_{j-1}|)}f(\alpha(\mathscr{X}_1),\cdots,\widehat{\alpha(\mathscr{X}_i)},\cdots,[\mathscr{X}_{i},\mathscr{X}_{j}]_{\alpha},\cdots,\alpha(\mathscr{X}_{m+1}),\alpha(z))\\
&+\sum_{i=1}^{m+1}(-1)^i(-1)^{|\mathscr{X}_i|(|\mathscr{X}_{i+1}|+\cdots+|\mathscr{X}_{m+1}|)}f(\alpha(\mathscr{X}_1),\cdots,\widehat{\alpha(\mathscr{X}_i)},\cdots,\alpha(\mathscr{X}_{m+1}),\mathscr{X}_i\cdot z)\\
& +\sum_{i=1}^{m+1}(-1)^{i+1}(-1)^{|\mathscr{X}_i|(|f|+|\mathscr{X}_{1}|+\cdots+|\mathscr{X}_{i-1}|)}\alpha^{m}(\mathscr{X}_i)\cdot f(\mathscr{X}_1,\cdots,\widehat{\mathscr{X}_i},\cdots,\mathscr{X}_{m+1}, z)\\
&  +(-1)^m(f(\mathscr{X}_1,\cdots,\mathscr{X}_m, ~~)\cdot \mathscr{X}_{m+1})\bullet_{\alpha} \alpha^{m}(z),
\end{align*}
where $\mathscr{X}_i=\mathscr{X}_i^1\wedge\cdots\wedge\mathscr{X}_i^{n-1}\in\mathfrak{g}^{\wedge^{n-1}}, i=1,\cdots,m+1, z\in\mathfrak{g}$ and the last term is defined by
\begin{equation}\begin{split}
(f(\mathscr{X}_1,\cdots,\mathscr{X}_m,~~ )\cdot \mathscr{X}_{m+1})\bullet_{\alpha} \alpha^{m}(z)=&\sum_{i=1}^{n-1}(-1)^{(|f|+|\mathscr{X}_{1}|+\cdots+|\mathscr{X}_{m}|)(|\mathscr{X}_{m+1}^1|+\cdots+|\mathscr{X}_{m+1}^{i-1}|)}\\
\cdot[\alpha^{m}(\mathscr{X}_{m+1}^1),\cdots,&f(\mathscr{X}_1,\cdots,\mathscr{X}_m,\mathscr{X}_{m+1}^i),\cdots,\alpha^{m}(\mathscr{X}_{m+1}^{n-1}),\alpha^{m}(z)].
\end{split}\end{equation}
\end{definition}
This cohomology complex is called adjoint  cohomology complex.
In parcular, by definition \ref{d6.1}, one gets the following definition:
\begin{definition}\label{d6.1} For $m\geq 1,$ we call $m$-coboundary operator of the $n$-ary multiplicative Hom-Nambu-Lie superalgebra $(\mathfrak{g}, [\cdot,\cdots,\cdot],\alpha)$
the even linear map $\delta^{m}: C^m(\mathfrak{g}, K)\rightarrow C^{m+1}(\mathfrak{g}, K)$ by
\begin{align*}
&(\delta^{m} f)(\mathscr{X}_1,\cdots,\mathscr{X}_m, \mathscr{X}_{m+1}, z)\\
=&\sum_{i<j}(-1)^i(-1)^{|\mathscr{X}_i|(|\mathscr{X}_{i+1}|+\cdots+|\mathscr{X}_{j-1}|)}f(\alpha(\mathscr{X}_1),\cdots,\widehat{\alpha(\mathscr{X}_i)},\cdots,[\mathscr{X}_{i},\mathscr{X}_{j}]_{\alpha},\cdots,\alpha(\mathscr{X}_{m+1}),\alpha(z))\\
&+\sum_{i=1}^{m+1}(-1)^i(-1)^{|\mathscr{X}_i|(|\mathscr{X}_{i+1}|+\cdots+|\mathscr{X}_{m+1}|)}f(\alpha(\mathscr{X}_1),\cdots,\widehat{\alpha(\mathscr{X}_i)},\cdots,\alpha(\mathscr{X}_{m+1}),\mathscr{X}_i\cdot z),
\end{align*}
where $\mathscr{X}_i=\mathscr{X}_i^1\wedge\cdots\wedge\mathscr{X}_i^{n-1}\in\mathfrak{g}^{\wedge^{n-1}}, i=1,\cdots,m+1, z\in\mathfrak{g}$.
\end{definition}
This cohomology complex is called scalar  cohomology complex.

\begin{definition} \label{d6.2} \cite{gbl}
Let $(\mathfrak{g}, [\cdot,\cdot,\cdot],\alpha)$ be a 3-ary multiplicative Hom-Lie superalgebra. $\mathscr{X}=x_{1}\wedge x_{2}\in \mathfrak{g}^{\wedge 2}$
is called as fundamental object of  $\mathfrak{g}$ and $\forall z\in \mathfrak{g}, \mathscr{X}\cdot z:=[x_{1},x_{2},z].$ Clearly, $|\mathscr{X}|=|x_{1}|+|x_{2}|.$

Let $\mathscr{X}=x_{1}\wedge x_{2}\in \mathfrak{g}^{\wedge 2}$ and $\mathscr{Y}=y_{1}\wedge y_{2}\in \mathfrak{g}^{\wedge 2}$ be two fundamental objects,
$[\cdot,\cdot]_{\alpha}: \mathfrak{g}^{\wedge 2}\times \mathfrak{g}^{\wedge 2}\rightarrow \mathfrak{g}^{\wedge 2}$ is a bilinear map, and it satisfies
$$[\cdot,\cdot]_{\alpha}=\mathscr{X}\cdot y_{1}\wedge \alpha(y_{2})+(-1)^{|\mathscr{X}||y_{1}|}\alpha(y_{1})\wedge \mathscr{X}\cdot y_{2}.$$
$\alpha:  \mathfrak{g}^{\wedge 2}\rightarrow \mathfrak{g}^{\wedge 2}$ is a linear map, and it satisfies $\alpha(\mathscr{X})=\alpha(x_{1})\wedge\alpha(x_{2}),$ then
by a direct computation one gets $\alpha[\mathscr{X},\mathscr{Y}]_{\alpha}=[\alpha(\mathscr{X}),\alpha(\mathscr{Y})]_{\alpha}.$
\end{definition}
By definition \ref{d6.1}, we also have
\begin{definition}
we call $2$-coboundary operator of the $3$-ary multiplicative Hom-Lie superalgebra $(\mathfrak{g}, [\cdot,\cdot,\cdot],\alpha)$
the even linear map $\delta^{2}: C^2(\mathfrak{g}, V)\rightarrow C^{3}(\mathfrak{g}, V)$ by
\begin{align*}
&(\delta^{2} f)(\mathscr{X},\mathscr{Y}, z)\\
=&-f([\mathscr{X},\mathscr{Y}]_{\alpha},\alpha(z))
-(-1)^{|\mathscr{X}||\mathscr{Y}|}f(\alpha(\mathscr{Y}),\mathscr{X}\cdot z)+f(\alpha(\mathscr{X}),  \mathscr{Y}\cdot z)
\\&-(f(\mathscr{X}, ~~)\bullet_{\alpha}\mathscr{Y})\cdot \alpha(z)-(-1)^{|\mathscr{Y}|(|\mathscr{X}|+|f|)}\alpha(\mathscr{Y})\cdot f(\mathscr{X}\cdot z)
+(-1)^{|\mathscr{X}||f|}\alpha(\mathscr{X})\cdot f(\mathscr{Y}\cdot z),
\end{align*}
where $\mathscr{X}=x_{1}\wedge x_{2}, \mathscr{Y}=y_{1}\wedge y_{2}\in\mathfrak{g}^{\wedge^{2}}, z\in\mathfrak{g},$
\begin{equation}\begin{split}
(f(\mathscr{X},~~ )\bullet_{\alpha}\mathscr{Y})\cdot \alpha(z)=&[f(\mathscr{X},y_{1}),\alpha(y_{2}), \alpha(z)]+(-1)^{(|f|+|\mathscr{X}|)|y_{1}|}[\alpha(y_{1}),f(\mathscr{X},y_{2}),\alpha(z)].
\end{split}\end{equation}
\end{definition}
\begin{theorem}\label{t6.5}
Let $(\mathfrak{g}, [\cdot,\cdot],\alpha)$ be a multiplicative Hom-Lie superalgebra, $\mathrm{str}\rho\circ \alpha=\mathrm{str}\rho$ and
$(\mathfrak{g}, [\cdot,\cdot,\cdot]_{\rho},\alpha)$ be the induced $3$-ary Hom-Lie superalgebra. Let $\varphi\in Z^{2}_{ad}(\mathfrak{g},\mathfrak{g}).$
Then $\varphi_{\rho}: \wedge^{2}\mathfrak{g}_{\rho}\wedge \mathfrak{g}\rightarrow \mathfrak{g}$ defined by
$$\varphi_{\rho}(\mathscr{X},z)=\mathrm{str}\rho(x_{1})\varphi(x_{2},z)-(-1)^{|x_{1}||x_{2}|}\mathrm{str}\rho(x_{2})\varphi(x_{1},z)+
(-1)^{|z|(|x_{1}|+|x_{2}|)}\mathrm{str}\rho(z)\varphi(x_{1},x_{2})$$
is a 2-cocycle of the induced 3-ary  Hom-Lie superalgebras, where for $\mathscr{X}=x_{1}\wedge x_{2}\in \wedge^{2}\mathfrak{g}_{\rho}.$
\end{theorem}
\begin{proof} Let $\varphi\in Z^{2}_{ad}(\mathfrak{g},\mathfrak{g}).$  Then we have
\begin{align*}
&(\delta^{2} \varphi_{\rho})(\mathscr{X},\mathscr{Y}, z)\\
=&-\varphi_{\rho}([\mathscr{X},\mathscr{Y}]_{\alpha},\alpha(z))
-(-1)^{|\mathscr{X}||\mathscr{Y}|}\varphi_{\rho}(\alpha(\mathscr{Y}),\mathscr{X}\cdot z)-(\varphi_{\rho}(\mathscr{X}, ~~)\bullet_{\alpha}\mathscr{Y})\cdot \alpha(z)
\\&+\varphi_{\rho}(\alpha(\mathscr{X}),  \mathscr{Y}\cdot z)-(-1)^{|\mathscr{Y}|(|\mathscr{X}|+|\varphi_{\rho}|)}\alpha(\mathscr{Y})\cdot \varphi_{\rho}(\mathscr{X}\cdot z)
+(-1)^{|\mathscr{X}||\varphi_{\rho}|}\alpha(\mathscr{X})\cdot \varphi_{\rho}(\mathscr{Y}\cdot z)
\\=&(-1)^{|y_{2}|(|x_{2}|+|y_{1}|)}\mathrm{str}\rho(x_{1})\mathrm{str}\rho(\alpha(y_{2}))\varphi([x_{2},y_{1}],\alpha(z))
\\&-(-1)^{|z|(|x_{2}|+|y_{1}|+|y_{2}|)}\mathrm{str}\rho(x_{1})\mathrm{str}\rho(\alpha(z))\varphi([x_{2},y_{1}],\alpha(y_{2}))
\\&-(-1)^{|x_{2}||x_{1}|+|y_{2}|(|x_{1}|+|y_{1}|)}\mathrm{str}\rho(x_{2})\mathrm{str}\rho(\alpha(y_{2}))\varphi([x_{1},y_{1}],\alpha(z))
\\&+(-1)^{|z|(|x_{1}|+|y_{1}|+|y_{2}|)+|x_{2}||x_{1}|}\mathrm{str}\rho(x_{2})\mathrm{str}\rho(\alpha(z))\varphi([x_{1},y_{1}],\alpha(y_{2}))
\\&-(-1)^{|y_{1}|(|x_{2}|+|x_{1}|)+|z|(|x_{1}|+|x_{2}|+|y_{2}|)}\mathrm{str}\rho(y_{1})\mathrm{str}\rho(\alpha(z))\varphi([x_{1},x_{2}],\alpha(y_{2}))
\\&-(-1)^{|\mathscr{X}||y_{1}|}\mathrm{str}\rho(x_{1})\mathrm{str}\rho(\alpha(y_{1}))\varphi([x_{2},y_{2}],\alpha(z))
\\&-(-1)^{|\mathscr{X}||y_{1}|+|z|(|y_{1}|+|y_{2}|+|x_{2}|)}\mathrm{str}\rho(x_{1})\mathrm{str}\rho(\alpha(z))\varphi(\alpha(y_{1}),[x_{2},y_{2}])
\\&+(-1)^{|\mathscr{X}||y_{1}|+|x_{1}||x_{2}|}\mathrm{str}\rho(x_{2})\mathrm{str}\rho(\alpha(y_{1}))\varphi([x_{1},y_{2}],\alpha(z))
\\&+(-1)^{|\mathscr{X}||y_{1}|+|x_{1}||x_{2}|+|z|(|y_{1}|+|y_{2}|+|x_{1}|)}\mathrm{str}\rho(x_{2})\mathrm{str}\rho(\alpha(z))\varphi(\alpha(y_{1}),[x_{1},y_{2}])
\\&-(-1)^{|\mathscr{X}||y_{1}|+|y_{2}|(|x_{1}|+|x_{2}|)+|z|(|x_{1}|+|x_{2}|+|y_{1}|)}\mathrm{str}\rho(y_{2})\mathrm{str}\rho(\alpha(z))
\varphi(\alpha(y_{1}),[x_{1},x_{2}])
\\&-(-1)^{|\mathscr{X}||\mathscr{Y}|}\mathrm{str}\rho(x_{1})\mathrm{str}\rho(\alpha(y_{1}))\varphi(\alpha(y_{2}),[x_{2},z])
\\&+(-1)^{|\mathscr{X}||\mathscr{Y}|+|y_{1}||y_{2}|}\mathrm{str}\rho(x_{1})\mathrm{str}\rho(\alpha(y_{2}))\varphi(\alpha(y_{1}),[x_{2},z])
\\&+(-1)^{|\mathscr{X}||\mathscr{Y}|+|x_{1}||x_{2}|}\mathrm{str}\rho(x_{2})\mathrm{str}\rho(\alpha(y_{1}))\varphi(\alpha(y_{2}),[x_{1},z])
\\&-(-1)^{|\mathscr{X}||\mathscr{Y}|+|x_{1}||x_{2}|+|y_{1}||y_{2}|}\mathrm{str}\rho(x_{2})\mathrm{str}\rho(\alpha(y_{2}))\varphi(\alpha(y_{1}),[x_{1},z])
\\&-(-1)^{|\mathscr{X}||\mathscr{Y}|+|z|(|x_{1}|+|x_{2}|)}\mathrm{str}\rho(z)\mathrm{str}\rho(\alpha(y_{1}))\varphi(\alpha(y_{2}),[x_{1},x_{2}])
\\&+(-1)^{|\mathscr{X}||\mathscr{Y}|+|z|(|x_{1}|+|x_{2}|)+|y_{1}||y_{2}|}\mathrm{str}\rho(z)\mathrm{str}\rho(\alpha(y_{2}))\varphi(\alpha(y_{1}),[x_{1},x_{2}])
\\&+(-1)^{|y_{2}|(|x_{2}|+|y_{1}|)}\mathrm{str}\rho(x_{1})\mathrm{str}\rho(\alpha(y_{2}))\varphi([x_{2},y_{1}],\alpha(z))
\\&-(-1)^{|z|(|x_{2}|+|y_{1}|+|y_{2}|)}\mathrm{str}\rho(x_{1})\mathrm{str}\rho(\alpha(z))\varphi([x_{2},y_{1}],\alpha(y_{2}))
\\&-(-1)^{|x_{2}||x_{1}|+|y_{2}|(|x_{1}|+|y_{1}|)}\mathrm{str}\rho(x_{2})\mathrm{str}\rho(\alpha(y_{2}))\varphi([x_{1},y_{1}],\alpha(z))
\\&+(-1)^{|z|(|x_{1}|+|y_{1}|+|y_{2}|)+|x_{2}||x_{1}|}\mathrm{str}\rho(x_{2})\mathrm{str}\rho(\alpha(z))\varphi([x_{1},y_{1}],\alpha(y_{2}))
\\&+(-1)^{(|y_{1}|+|y_{2}|)(|x_{2}|+|x_{1}|)}\mathrm{str}\rho(y_{1})\mathrm{str}\rho(\alpha(y_{2}))\varphi([x_{1},x_{2}],\alpha(z))
\\&-(-1)^{|y_{1}|(|x_{2}|+|x_{1}|)+|z|(|y_{2}|+|x_1|+|x_{2}|)}\mathrm{str}\rho(y_{1})\mathrm{str}\rho(\alpha(z))\varphi([x_{1},x_{2}],\alpha(y_{2}))
\\&-(-1)^{(|\varphi_{\rho}|+|\mathscr{X}|)|y_{1}|}\mathrm{str}\rho(x_{1})\mathrm{str}\rho(\alpha(y_{1}))\varphi([x_{2},y_{2}],\alpha(z))
\\&-(-1)^{(|\varphi_{\rho}|+|\mathscr{X}|)|y_{1}|+|z|(|y_{1}|+|y_{2}|+|x_{2}|)}\mathrm{str}\rho(x_{1})\mathrm{str}\rho(\alpha(z))\varphi(\alpha(y_{1}),[x_{2},y_{2}])
\\&+(-1)^{(|\varphi_{\rho}|+|\mathscr{X}|)|y_{1}|+|x_{1}||x_{2}|}\mathrm{str}\rho(x_{2})\mathrm{str}\rho(\alpha(y_{1}))\varphi([x_{1},y_{2}],\alpha(z))
\\&+(-1)^{(|\varphi_{\rho}|+|\mathscr{X}|)|y_{1}|+|x_{1}||x_{2}|+|z|(|y_{1}|+|y_{2}|+|x_{1}|)}\mathrm{str}\rho(x_{2})\mathrm{str}\rho(\alpha(z))
\varphi(\alpha(y_{1}),[x_{1},y_{2}])
\\&-(-1)^{(|\varphi_{\rho}|+|\mathscr{X}|)|y_{1}|+|y_{2}|(|x_{1}|+|x_{2}|)}\mathrm{str}\rho(y_{2})\mathrm{str}\rho(\alpha(y_{1}))
\varphi([x_{1},x_{2}],\alpha(z))
\\&-(-1)^{(|\varphi_{\rho}|+|\mathscr{X}|)|y_{1}|+|y_{2}|(|x_{1}|+|x_{2}|)+|z|(|y_{1}|+|x_{2}|+|x_{1}|)}\mathrm{str}\rho(y_{2})\mathrm{str}\rho(\alpha(z))
\varphi(\alpha(y_{1}),[x_{1},x_{2}])
\\&+\mathrm{str}\rho(y_{1})\mathrm{str}\rho(\alpha(x_{1}))\varphi(\alpha(x_{2}),[y_{2},z])
\\&-(-1)^{|x_{1}||x_{2}|}\mathrm{str}\rho(y_{1})\mathrm{str}\rho(\alpha(x_{2}))\varphi(\alpha(x_{1}),[y_{2},z])
\\&-(-1)^{|y_{1}||y_{2}|}\mathrm{str}\rho(y_{2})\mathrm{str}\rho(\alpha(x_{1}))\varphi(\alpha(x_{2}),[y_{1},z])
\\&+(-1)^{|y_{1}||y_{2}|+|x_{1}||x_{2}|}\mathrm{str}\rho(y_{2})\mathrm{str}\rho(\alpha(x_{2}))\varphi(\alpha(x_{1}),[y_{1},z])
\\&+(-1)^{|z|(|y_{1}|+|y_{2}|)}\mathrm{str}\rho(z)\mathrm{str}\rho(\alpha(x_{1}))\varphi(\alpha(x_{2}),[y_{1},y_{2}])
\\&-(-1)^{|z|(|y_{1}|+|y_{2}|)+|x_{1}||x_{2}|}\mathrm{str}\rho(z)\mathrm{str}\rho(\alpha(x_{2}))\varphi(\alpha(x_{1}),[y_{1},y_{2}])
\\&-(-1)^{|\mathscr{Y}|(|\mathscr{X}|+|\varphi_{\rho}|)}\mathrm{str}\rho(x_{1})\mathrm{str}\rho(\alpha(y_{1}))\varphi(\alpha(y_{2}),[x_{2},z])
\\&+(-1)^{|\mathscr{Y}|(|\mathscr{X}|+|\varphi_{\rho}|)+|y_{1}||y_{2}|}\mathrm{str}\rho(x_{1})\mathrm{str}\rho(\alpha(y_{2}))\varphi(\alpha(y_{1}),[x_{2},z])
\\&+(-1)^{|\mathscr{Y}|(|\mathscr{X}|+|\varphi_{\rho}|)+|x_{1}||x_{2}|}\mathrm{str}\rho(x_{2})\mathrm{str}\rho(\alpha(y_{1}))
\varphi(\alpha(y_{2}),[x_{1},z])
\\&-(-1)^{|\mathscr{Y}|(|\mathscr{X}|+|\varphi_{\rho}|)+|x_{1}||x_{2}|+|y_{1}||y_{2}|}\mathrm{str}\rho(x_{2})\mathrm{str}\rho(\alpha(y_{2}))
\varphi(\alpha(y_{1}),[x_{1},z])
\\&-(-1)^{|\mathscr{Y}|(|\mathscr{X}|+|\varphi_{\rho}|)+|z|(|x_{1}|+|x_{2}|)}\mathrm{str}\rho(z)\mathrm{str}\rho(\alpha(y_{1}))\varphi(\alpha(y_{2}),[x_{1},x_{2}])
\\&+(-1)^{|\mathscr{Y}|(|\mathscr{X}|+|\varphi_{\rho}|)+|z|(|x_{1}|+|x_{2}|)+|y_{1}||y_{2}|}\mathrm{str}\rho(z)\mathrm{str}\rho(\alpha(y_{2}))
\varphi(\alpha(y_{1}),[x_{1},x_{2}])
\\&+(-1)^{|\mathscr{X}||\varphi_{\rho}|}\mathrm{str}\rho(y_{1})\mathrm{str}\rho(\alpha(x_{1}))\varphi(\alpha(x_{2}),[y_{2},z])
\\&-(-1)^{|\mathscr{X}||\varphi_{\rho}|+|x_{1}||x_{2}|}\mathrm{str}\rho(y_{1})\mathrm{str}\rho(\alpha(x_{2}))\varphi(\alpha(x_{1}),[y_{2},z])
\\&-(-1)^{|\mathscr{X}||\varphi_{\rho}|+|y_{1}||y_{2}|}\mathrm{str}\rho(y_{2})\mathrm{str}\rho(\alpha(x_{1}))
\varphi(\alpha(x_{2}),[y_{1},z])
\\&+(-1)^{|\mathscr{X}||\varphi_{\rho}|+|y_{1}||y_{2}|+|x_{1}||x_{2}|}\mathrm{str}\rho(y_{2})\mathrm{str}\rho(\alpha(x_{2}))
\varphi(\alpha(x_{1}),[y_{1},z])
\\&+(-1)^{|\mathscr{X}||\varphi_{\rho}|+|z|(|y_{1}|+|y_{2}|)}\mathrm{str}\rho(z)\mathrm{str}\rho(\alpha(x_{1}))\varphi(\alpha(x_{2}),[y_{1},y_{2}])
\\&-(-1)^{|\mathscr{X}||\varphi_{\rho}|+|z|(|y_{1}|+|y_{2}|)+|x_{1}||x_{2}|}\mathrm{str}\rho(z)\mathrm{str}\rho(\alpha(x_{2}))
\varphi(\alpha (x_{1}),[y_{1},y_{2}])
\\=&2(-1)^{|y_{2}|(|x_{2}|+|y_{1}|)}\mathrm{str}\rho(x_{1})\mathrm{str}\rho(\alpha(y_{2}))\varphi([x_{2},y_{1}],\alpha(z))
\\&-2(-1)^{|z|(|x_{2}|+|y_{1}|+|y_{2}|)}\mathrm{str}\rho(x_{1})\mathrm{str}\rho(\alpha(z))\varphi([x_{2},y_{1}],\alpha(y_{2}))
\\&-2(-1)^{|x_{2}||x_{1}|+|y_{2}|(|x_{1}|+|y_{1}|)}\mathrm{str}\rho(x_{2})\mathrm{str}\rho(\alpha(y_{2}))\varphi([x_{1},y_{1}],\alpha(z))
\\&+2(-1)^{|z|(|x_{1}|+|y_{1}|+|y_{2}|)+|x_{2}||x_{1}|}\mathrm{str}\rho(x_{2})\mathrm{str}\rho(\alpha(z))\varphi([x_{1},y_{1}],\alpha(y_{2}))
\\&-2(-1)^{(|y_{1}|+|y_{2}|)(|x_{2}|+|x_{1}|)}\mathrm{str}\rho(x_{1})\mathrm{str}\rho(\alpha(y_{1}))\varphi(\alpha(y_{2}),[x_{2},z])
\\&+2(-1)^{(|y_{1}|+|y_{2}|)(|x_{2}|+|x_{1}|)+|y_{2}||y_1|}\mathrm{str}\rho(x_{1})\mathrm{str}\rho(\alpha(y_{2}))\varphi(\alpha(y_{1}),[x_{2},z])
\\&+2(-1)^{|\mathscr{X}||\mathscr{Y}|+|x_{1}||x_{2}|}\mathrm{str}\rho(x_{2})\mathrm{str}\rho(\alpha(y_{1}))\varphi(\alpha(y_{2}),[x_{1},z])
\\&-2(-1)^{|\mathscr{X}||\mathscr{Y}|+|x_{1}||x_{2}|+|y_{1}||y_{2}|}\mathrm{str}\rho(x_{2})\mathrm{str}\rho(\alpha(y_{2}))\varphi(\alpha(y_{1}),[x_{1},z])
\\&+2(-1)^{|\mathscr{X}||\mathscr{Y}|+|z|(|x_{1}|+|x_{2}|)+|y_{1}||y_{2}|}\mathrm{str}\rho(z)\mathrm{str}\rho(\alpha(y_{2}))\varphi(\alpha(y_{1}),[x_{1},x_{2}])
\\&+2(-1)^{|\mathscr{X}||\varphi_{\rho}|}\mathrm{str}\rho(y_{1})\mathrm{str}\rho(\alpha(x_{1}))\varphi(\alpha(x_{2}),[y_{2},z])
\\&-2(-1)^{|\mathscr{X}||\varphi_{\rho}|+|x_{1}||x_{2}|}\mathrm{str}\rho(y_{1})\mathrm{str}\rho(\alpha(x_{2}))\varphi(\alpha(x_{1}),[y_{2},z])
\\&-2(-1)^{|\mathscr{X}||\varphi_{\rho}|+|y_{1}||y_{2}|}\mathrm{str}\rho(y_{2})\mathrm{str}\rho(\alpha(x_{1}))\varphi(\alpha(x_{2}),[y_{1},z])
\\&+2(-1)^{|\mathscr{X}||\varphi_{\rho}|+|y_{1}||y_{2}|+|x_{1}||x_{2}|}\mathrm{str}\rho(y_{2})\mathrm{str}\rho(\alpha(x_{2}))
\varphi(\alpha(x_{1}),[y_{1},z])
\\&+2(-1)^{|\mathscr{X}||\varphi_{\rho}|+|z|(|y_{1}|+|y_{2}|)}\mathrm{str}\rho(z)\mathrm{str}\rho(\alpha(x_{1}))\varphi(\alpha(x_{2}),[y_{1},y_{2}])
\\&-2(-1)^{|\mathscr{X}||\varphi_{\rho}|+|z|(|y_{1}|+|y_{2}|)+|x_{1}||x_{2}|}\mathrm{str}\rho(z)\mathrm{str}\rho(\alpha(x_{2}))
\varphi(\alpha (x_{1}),[y_{1},y_{2}])
\\&-2(-1)^{(|\varphi_{\rho}|+|\mathscr{X}|)|y_{1}|}\mathrm{str}\rho(x_{1})\mathrm{str}\rho(\alpha(y_{1}))\varphi([x_{2},y_{2}],\alpha(z))
\\&-2(-1)^{(|\varphi_{\rho}|+|\mathscr{X}|)|y_{1}|+|z|(|y_{1}|+|y_{2}|+|x_{2}|)}\mathrm{str}\rho(x_{1})\mathrm{str}\rho(\alpha(z))\varphi(\alpha(y_{1}),[x_{2},y_{2}])
\\&+2(-1)^{(|\varphi_{\rho}|+|\mathscr{X}|)|y_{1}|+|x_{1}||x_{2}|}\mathrm{str}\rho(x_{2})\mathrm{str}\rho(\alpha(y_{1}))\varphi([x_{1},y_{2}],\alpha(z))
\\&+2(-1)^{(|\varphi_{\rho}|+|\mathscr{X}|)|y_{1}|+|x_{1}||x_{2}|+|z|(|y_{1}|+|y_{2}|+|x_{1}|)}\mathrm{str}\rho(x_{2})\mathrm{str}\rho(\alpha(z))
\varphi(\alpha(y_{1}),[x_{1},y_{2}])
\\&-2(-1)^{(|\varphi_{\rho}|+|\mathscr{X}|)|y_{1}|+|y_{2}|(|x_{1}|+|x_{2}|)+|z|(|y_{1}|+|x_{2}|+|x_{1}|)}\mathrm{str}\rho(y_{2})\mathrm{str}\rho(\alpha(z))
\varphi(\alpha(y_{1}),[x_{1},x_{2}])
\\&=0,
\end{align*}
by $\mathrm{str}\rho[x,y]=0,$ and  $\delta^{2}\varphi(x,y,z)=(-1)^{|x|(|y|+|z|)}\varphi(\alpha(y),[z,x])+(-1)^{|z|(|x|+|y|)}\varphi(\alpha(z),[x,$\\$y])+\varphi(\alpha(x),[y,z])=0,$ since $\varphi\in Z^{2}_{ad}(\mathfrak{g},\mathfrak{g}).$
\end{proof}
In particular, by Theorem \ref{t6.5}, we have the following corollary:
\begin{corollary}
Let $(\mathfrak{g}, [\cdot,\cdot],\alpha)$ be a multiplicative Hom-Lie superalgebra, $\mathrm{str}\rho\circ \alpha=\mathrm{str}\rho$ and
$(\mathfrak{g}, [\cdot,\cdot,\cdot]_{\rho},\alpha)$ be the induced $3$-ary Hom-Lie superalgebra. Let $\varphi\in Z^{2}_{0}(\mathfrak{g},K).$
Then $\varphi_{\rho}: \wedge^{2}\mathfrak{g}_{\rho}\wedge \mathfrak{g}\rightarrow K$ defined by
$$\varphi_{\rho}(\mathscr{X},z)=\mathrm{str}\rho(x_{1})\varphi(x_{2},z)-(-1)^{|x_{1}||x_{2}|}\mathrm{str}\rho(x_{2})\varphi(x_{1},z)+
(-1)^{|z|(|x_{1}|+|x_{2}|)}\mathrm{str}\rho(z)\varphi(x_{1},x_{2})$$
is a 2-cocycle of the induced 3-ary  Hom-Lie superalgebras, where for $\mathscr{X}=x_{1}\wedge x_{2}\in \wedge^{2}\mathfrak{g}_{\rho}.$
\end{corollary}

\begin{proposition}
Every 1-cocycle for the scalar cohomology of a multiplicative Hom-Lie superalgebra $(\mathfrak{g}, [\cdot,\cdot],\alpha)$
is a 1-cocycle for the scalar cohomology of the induced algebra.
\end{proposition}

\begin{proof} Let $\omega$ be a 1-cocycle for the scalar cohomology of $(\mathfrak{g}, [\cdot,\cdot],\alpha).$  Then
$$\forall x_{1}, z\in \mathfrak{g},  \, \delta^{1}\omega(x_{1}, z)=-\omega([x_{1}, z])=0,$$
which is equivalent to $[\mathfrak{g},\mathfrak{g}]\subset \mathrm{ker}\omega,$ by Remark \ref{re4.13}, $[\mathfrak{g},\mathfrak{g},\mathfrak{g}]_{\rho}\subset [\mathfrak{g},\mathfrak{g}],$ then $[\mathfrak{g},\mathfrak{g},\mathfrak{g}]_{\rho}\subset \mathrm{ker}\omega,$ that is,
$$\forall x_{1}, x_{2}, z\in \mathfrak{g},  \, \omega([x_{1}, x_{2}, z]_{\rho})=-\delta^{1}\omega(x_{1}, x_{2}, z)=0.$$
It means that $\omega$ is a 1-cocycle for the scalar cohomology of $(\mathfrak{g}, [\cdot,\cdot,\cdot]_{\rho},\alpha).$
\end{proof}

\begin{lemma} Let $\omega\in C^{1}(\mathfrak{g},K).$ Then
\begin{align*}
\delta_{\rho}^{1}\omega(x_{1},x_{2},x_{3})&=\mathrm{str}\rho(x_{1})\delta^{1}\omega(x_{2},x_{3})-(-1)^{|x_{1}||x_{2}|}
\mathrm{str}\rho(x_{2})\delta^{1}\omega(x_{1},x_{3})
\\&+
(-1)^{|x_{3}|(|x_{1}|+|x_{2}|)}\mathrm{str}\rho(x_{3})\delta^{1}\omega(x_{1},x_{2}),  \forall x_{1},x_{2},x_{3}\in\mathfrak{g}.
\end{align*}
where $\delta_{\rho}^{p}$ is the coboundary operator for the cohomology  complex of the induced algebra $\mathfrak{g}_{\rho}.$
\end{lemma}

\begin{proof} Let $\omega\in C^{1}(\mathfrak{g},K), x_{1},x_{2},x_{3}\in\mathfrak{g},$  then
\begin{align*}
&\delta_{\rho}^{1}\omega(x_{1},x_{2},x_{3})=
-\omega([x_{1},x_{2},x_{3}]_{\rho})=-
\mathrm{str}\rho(x_{1})\omega(x_{2},x_{3})\\&+(-1)^{|x_{1}||x_{2}|}
\mathrm{str}\rho(x_{2})\omega(x_{1},x_{3})
-(-1)^{|x_{3}|(|x_{1}|+|x_{2}|)}\mathrm{str}\rho(x_{3})\omega(x_{1},x_{2})
\\=& \mathrm{str}\rho(x_{1})\delta^{1}\omega(x_{2},x_{3})-(-1)^{|x_{1}||x_{2}|}
\mathrm{str}\rho(x_{2})\delta^{1}\omega(x_{1},x_{3})
+(-1)^{|x_{3}|(|x_{1}|+|x_{2}|)}\mathrm{str}\rho(x_{3})\delta^{1}\omega(x_{1},x_{2}).
\end{align*}
\end{proof}

\begin{proposition} Let $\varphi_{1}, \varphi_{2}\in Z^{2}_{0}(\mathfrak{g},K).$ If $\varphi_{1}, \varphi_{2}$ are in the same cohomology class, then
$\psi_{1}, \psi_{2}$ defined by
\begin{align*}
\psi_{i}(x_{1},x_{2},x_{3})&=\mathrm{str}\rho(x_{1})\varphi_{i}(x_{2},x_{3})-(-1)^{|x_{1}||x_{2}|}
\mathrm{str}\rho(x_{2})\varphi_{i}(x_{1},x_{3})
\\&+
(-1)^{|x_{3}|(|x_{1}|+|x_{2}|)}\mathrm{str}\rho(x_{3})\varphi_{i}(x_{1},x_{2}), i=1,2,
\end{align*}
are  in the same cohomology class.
\end{proposition}

\begin{proof} Let $\varphi_{1}, \varphi_{2}\in Z^{2}_{0}(\mathfrak{g},K)$ be two cocycles in the same cohomology class, that is
$\varphi_{2}-\varphi_{1}\in \delta^{1}\omega, \omega\in C^{1}(\mathfrak{g},K),$
and \begin{align*}
\psi_{i}(x_{1},x_{2},x_{3})&=\mathrm{str}\rho(x_{1})\varphi_{i}(x_{2},x_{3})-(-1)^{|x_{1}||x_{2}|}
\mathrm{str}\rho(x_{2})\varphi_{i}(x_{1},x_{3})
\\&+
(-1)^{|x_{3}|(|x_{1}|+|x_{2}|)}\mathrm{str}\rho(x_{3})\varphi_{i}(x_{1},x_{2}), i=1,2.
\end{align*}
Then we have
\begin{align*}
&\psi_{2}(x_{1},x_{2},x_{3})-\psi_{1}(x_{1},x_{2},x_{3})\\=&
\mathrm{str}\rho(x_{1})\varphi_{2}(x_{2},x_{3})-(-1)^{|x_{1}||x_{2}|}
\mathrm{str}\rho(x_{2})\varphi_{2}(x_{1},x_{3})+(-1)^{|x_{3}|(|x_{1}|+|x_{2}|)}\mathrm{str}\rho(x_{3})\varphi_{2}(x_{1},x_{2})
\\&-\mathrm{str}\rho(x_{1})\varphi_{1}(x_{2},x_{3})+(-1)^{|x_{1}||x_{2}|}
\mathrm{str}\rho(x_{2})\varphi_{1}(x_{1},x_{3})-(-1)^{|x_{3}|(|x_{1}|+|x_{2}|)}\mathrm{str}\rho(x_{3})\varphi_{1}(x_{1},x_{2})
\\=&\mathrm{str}\rho(x_{1})(\varphi_{2}-\varphi_{1})(x_{2},x_{3})-(-1)^{|x_{1}||x_{2}|}
\mathrm{str}\rho(x_{2})(\varphi_{2}-\varphi_{1})(x_{1},x_{3})
\\&+(-1)^{|x_{3}|(|x_{1}|+|x_{2}|)}\mathrm{str}\rho(x_{3})(\varphi_{2}-\varphi_{1})(x_{1},x_{2})
\\=&\mathrm{str}\rho(x_{1})\delta^{1}\omega(x_{2},x_{3})-(-1)^{|x_{1}||x_{2}|}
\mathrm{str}\rho(x_{2})\delta^{1}\omega(x_{1},x_{3})
+(-1)^{|x_{3}|(|x_{1}|+|x_{2}|)}\mathrm{str}\rho(x_{3})\delta^{1}\omega(x_{1},x_{2})
\\=&\delta_{\rho}^{1}\omega(x_{1},x_{2},x_{3}).
\end{align*}
That means that $\psi_{1}, \psi_{2}$ are in the same cohomology class.
\end{proof}

\end{document}